\documentclass{amsart}
\usepackage{amssymb,amsthm,amstext,latexsym,amsmath,fullpage}
\usepackage{enumitem}
\usepackage[all]{xy}
\usepackage{tensor}
\usepackage[colorlinks=true]{hyperref}
\usepackage{stmaryrd}
\hypersetup{urlcolor=blue, citecolor=red}

\newcommand{\asm}{(M,\sigma)}
\newcommand{\cotan}{\mathrm{T}^*}
\newcommand{\de}{\mathrm{d}}
\newcommand{\R}{\mathbb{R}}
\newcommand{\tpm}{(P,\pi,\phi)}
\newcommand{\ir}{p:\asm \to \tpm}
\newcommand{\Z}{\mathbb{Z}}

\theoremstyle{definition}

\newtheorem{defn}{Definition}

\newtheorem{example}{Example}

\newtheorem{obs}{Observation}
\newtheorem*{qn}{Question}

\theoremstyle{plain}
\newtheorem{theorem}{Theorem}
\newtheorem{proposition}{Proposition}
\newtheorem{corollary}{Corollary}
\newtheorem{lemma}{Lemma}

\theoremstyle{remark}
\newtheorem{remark}{Remark}

\title[Twisted isotropic realisations]{Twisted isotropic realisations of twisted Poisson structures}

\author[Nicola Sansonetto and Daniele Sepe]{}



\email{nicola.sansonetto@gmail.com}
\email{dsepe@math.ist.utl.pt}

\thanks{D.S. was partially supported by the Laura Wisewell Travel Grant of the University of Edinburgh, by the Leicester Mathematical Fellowship of the University of Leicester, and by FCT grant SFRH/BPD/77263/2011. N.S. was partially supported by grant AdR1249109 ``Dinamica anolonoma'' from Universit\`a degli Studi di Verona.}

\begin{document}
\begin{abstract}
Motivated by the recent connection between nonholonomic integrable systems and twisted Poisson manifolds made in \cite{balseiro_garcia_naranjo}, this paper investigates the global theory of integrable Hamiltonian systems on almost symplectic manifolds as an initial step to understand Hamiltonian integrability on twisted Poisson (and Dirac) manifolds. Non-commutative integrable Hamiltonian systems on almost symplectic manifolds were first defined in \cite{fasso_sansonetto}, which proved existence of local generalised action-angle coordinates in the spirit of the Liouville-Arnol'd theorem. In analogy with their symplectic counterpart, these systems can be described globally by twisted isotropic realisations of twisted Poisson manifolds, a special case of symplectic realisations of twisted Dirac structures considered in \cite{bursztyn_crainic_weinstein_zhu}. This paper classifies twisted isotropic realisations up to smooth isomorphism and provides a cohomological obstruction to the construction of these objects, generalising the main results of \cite{daz_delz}.
\end{abstract}

\maketitle

\centerline{\scshape Nicola Sansonetto}
\medskip
{\footnotesize
  \centerline{Dipartimento di Informatica, Universit\`a degli Studi di Verona}
  \centerline{Ca' Vignal 2, Strada Le Grazie 15, 37134 Verona, Italy}
  \centerline{}
}
\medskip

\centerline{\scshape Daniele Sepe}
\medskip
{\footnotesize
 \centerline{CAMGSD, Instituto Superior T\'ecnico}
   \centerline{Av. Rovisco Pais}
   \centerline{Lisboa, 1049-001, Portugal}
}

\bigskip

\section{Introduction}\label{sec:introduction}
Let $P$ be a smooth manifold, let $\pi,\phi$ be a smooth bivector and a closed 3-form on $P$ respectively, and denote by $\{.,.\}$ the almost Poisson bracket on $\mathrm{C}^{\infty}(P)$ induced by $\pi$. Say that $(\pi,\phi)$ defines a \emph{twisted Poisson structure} on $P$ if $\phi$ measures the failure of $\{.,.\}$ to satisfy the Jacobi identity; in this case the triple $\tpm$ is called a \emph{twisted Poisson manifold} (cf. Definition \ref{defn:twisted_poisson}). Interest in twisted Poisson manifolds originally stemmed from string theory (cf. \cite{severa_weinstein}), but such objects have become of independent mathematical interest as they retain many of the geometrical properties of Poisson manifolds (cf. \cite{bursztyn_crainic_weinstein_zhu,cattaneo_xu}). In particular, Poisson manifolds are to symplectic geometry as twisted Poisson manifolds are to \emph{almost} symplectic geometry, \textit{i.e.} the geometry of manifolds endowed with a non-degenerate 2-form. \\

Very recently, twisted Poisson manifolds have been linked to certain nonholonomic mechanical systems, most notably the Veselova system and the Chaplygin sphere (cf. \cite{balseiro_garcia_naranjo}). In recent years the interest in such systems and their applications has been growing, as these are examples of non-Hamiltonian integrable systems, \textit{i.e.} the dynamics is quasi-periodic along tori (or cylinders). Thus far there are two different approaches to the study of the integrability outside the (symplectic or Poisson) Hamiltonian framework: 

\begin{enumerate}[label=\arabic*), ref=\arabic*]
\item \label{item:28} reconstruction from reduced periodic (and recently quasi--periodic) dynamics (cf. \cite{ashwin_melbourne,bloch_zenkov,FGNS,FG07,hermans,zenkov});
\item \label{item:38} possible generalisations of the Liouville-Arnol'd Theorem or its non-commutative counterpart both in symplectic and Poisson geometry (cf. \cite{Bog,bates_cushman,CD,fasso,lgmv}). 
\end{enumerate} 

As for approach \ref{item:28}, many nonholonomic systems have a symmetry group which allows to consider dynamics on the reduced phase space; in many cases, the reduced dynamics is periodic or quasi--periodic along tori (or cylinders). When the reduced dynamics is periodic and if the symmetry group is compact, the reconstructed dynamics on the original phase space is quasi--periodic on tori of a certain dimension (at least in an open and dense set) and, thus, integrable; little is known if the symmetry group is not compact or the motion in the reduced phase space is quasi--periodic (cf. the forthcoming \cite{FGNS} for some results in this direction). In the case of the Veselova system or the Chaplygin sphere, \cite{balseiro_garcia_naranjo} proves that the dynamics on the reduced phase space can be formulated in terms of twisted Poisson structures. This begs the following intriguing question.

\begin{qn}\label{qn:complete_integrable}
  Are the reduced equations of motion for the Veselova system or the Chaplygin sphere \emph{completely integrable} with respect to the twisted Poisson structures constructed in \cite{balseiro_garcia_naranjo}?
\end{qn}

This would give a direct proof of the integrability of the reduced systems without a need to use what is known as time-reparametrisation. The above question naturally invites to study a generalisation of the Liouville-Arnol'd theorem for Hamiltonian integrable systems on twisted Poisson structures, which would give a \emph{local} description of such systems in terms of (generalised) action-angle coordinates, in line with approach \ref{item:38} above. The first step to understand this case is to consider the simplest possible scenario of a Hamiltonian integrable system defined on an almost symplectic manifold; this work was carried out in \cite{fasso_sansonetto}, where, under a crucial assumption on the integrals of motion (cf. Definitions \ref{defn:strongly_hamiltonian} and \ref{defn:ncias}), an analogue of the Liouville-Arnol'd theorem is proved (cf. Theorem \ref{thm:lma_almost}). \\

The aim of this paper is to develop the \emph{global} theory of non-commutative integrable Hamiltonian systems (NCIHS for short) on almost symplectic manifolds in the sense of \cite{fasso_sansonetto}, generalising results of \cite{daz_delz,dui}. The crucial starting point is that if $F=(f_1,\ldots,f_{2n-k}):\asm \to \R^{2n-k}$ is such a system then $P =F(M)$ inherits a twisted Poisson structure $(\pi,\phi)$ with respect to which $F:\asm \to \tpm$ is a \emph{twisted isotropic realisation}, \textit{i.e.} a twisted Poisson morphism whose fibres are isotropic submanifolds (cf. Lemma \ref{lemma:twisted_base} and Definition \ref{defn:symplectic_realisations_almost_poisson}). This generalises the correspondence between completely integrable Hamiltonian systems in the sense of \cite{mishchenko_fomenko,nek} and isotropic realisations of Poisson manifolds (cf. \cite{daz_delz}). As expected, all results that can be obtained in the Poisson context admit analogues in the twisted Poisson case; in particular, the following hold for any fixed twisted isotropic realisation $\ir$ in the sense of Definition \ref{defn:symplectic_realisations_almost_poisson} (cf. \cite{daz_delz,vaisman} for the corresponding statements for isotropic realisations):

\begin{itemize}
\item $\tpm$ is regular (cf. Definition \ref{defn:regular_twisted_Poisson} and Proposition \ref{prop:isotropic_implies_regular});
\item it is classified, up to smooth isomorphism, by two invariants: its \emph{period net} $\Lambda \to P$ (cf. Definition \ref{defn:period_net}) and its \emph{Chern class} $c \in \mathrm{H}^2(P;\mathcal{P}_{\Lambda})$, where $\mathcal{P}_{\Lambda}$ denotes the sheaf of sections of the period net (cf. Definition \ref{defn:cc});
\item the almost symplectic foliation $\mathcal{F}$ of $\tpm$ admits a transversally integral affine structure $\mathcal{A}$ which can be identified with the period net $\Lambda \to P$ (cf. Definition \ref{defn:tias}, Proposition \ref{prop:tias_equiv_period_net} and Corollary \ref{cor:tias_twisted_ir}).
\end{itemize}

The main result of this note is Theorem \ref{thm:main}, which generalises the main result of \cite{daz_delz}, gives a cohomological criterion to \emph{constructing} twisted isotropic realisations. Namely, given a regular twisted Poisson manifold $\tpm$ with a given transversally integral affine structure $\Lambda \to P$, a necessary and sufficient condition for $c \in \mathrm{H}^2(P;\mathcal{P}_{\Lambda})$ to be the Chern class of some twisted isotropic realisation of $\tpm$ is that 
$$ \mathcal{D}_{\Lambda}(c) = \xi_{\phi}, $$
\noindent
where $\mathcal{D}_{\Lambda}$ is the \emph{Dazord-Delzant} homomorphism associated to the transversally integral affine structure $\Lambda \to P$ (cf. \cite{daz_delz} and Section \ref{sec:affine-geometry}) and $\xi_{\phi}$ is the \emph{characteristic class} of $\tpm$ (cf. Definition \ref{defn:characteristic_class}). It is important to remark that transversally integral affine geometry plays a crucial role in the constructions associated to the proof of Theorem \ref{thm:main}, which is further indication of its importance in the study of (twisted) isotropic realisations and, more generally, of Hamiltonian integrable systems. \\

This paper is structured as follows. Section \ref{sec:twist-poiss-manif} recalls the basic properties of twisted Poisson manifolds with particular focus on the regular case. The notion of NCIHS on almost symplectic manifolds is introduced in Section \ref{sec:integr-almost-sympl}, along with a statement of the Liouville-Arnol'd theorem in this case (cf. \cite{fasso_sansonetto}). Section \ref{sec:relat-twist-poiss} proves the relation between twisted isotropic realisations and NCIHS, extending some results obtained in \cite{fasso_sansonetto}. The smooth classification of twisted isotropic realisations is carried out in Section \ref{sec:twist-isotr-real} using the theory of principal groupoid bundles (cf. \cite{mackenzie,rossi}); moreover, existence of local action-angle coordinates is reproved using intrinsic methods (cf. Proposition \ref{prop:almost_symplectic_pullback}). Section \ref{sec:affine-geometry} proves the existence of a transversally integral affine structure on the base of a twisted isotropic realisation. Such a structure is used to solve the construction problem of twisted isotropic realisations (cf. Section \ref{sec:constr-isotr-real} and Theorem \ref{thm:main}). Section \ref{sec:conclusion} contains further questions stemming from this work, which are going to be considered in future works. Throughout this note, there are two types of comments, labelled \textbf{Observation} and \textit{Remark} respectively; those with the former label are central to the problems studied in this paper, while those with the latter may be skipped at a first reading.

\section{Twisted Poisson manifolds}\label{sec:twist-poiss-manif}
In this section twisted Poisson manifolds are introduced and some of their basic properties are recalled in order to prepare the ground for the work in Section \ref{sec:twist-isotr-real} and to establish notation. Most of the material presented below can also be found in \cite{balseiro_garcia_naranjo,cattaneo_xu,severa_weinstein}. \\

Recall that an \emph{almost Poisson structure} on a smooth manifold $P$ is a bracket $\{.,.\}$ on $\mathrm{C}^{\infty}(P)$ which is skew-symmetric, bilinear and satisfies the Leibniz rule, \textit{i.e.} for all $h_1,h_2,h_3 \in \mathrm{C}^{\infty}(P)$, 
$$ \{h_1,h_2h_3\} = \{h_1,h_2\}h_3 + h_2\{h_1,h_3\}. $$
\noindent
Associated to an almost Poisson structure $\{.,.\}$ on a manifold $P$ is a bivector field $\pi \in \Gamma (\Lambda^2 \mathrm{T}P)$ uniquely defined by
$$ \pi(\de h_1, \de h_2) := \{h_1,h_2\} $$
\noindent 
and extended by $\mathrm{C}^{\infty}(P)$-linearity. Given a bivector field $\pi$ on $P$, denote by $\pi^{\#}: \cotan P \to \mathrm{T}P$ the map defined by 
$$ (\pi^{\#}(\alpha))(\beta) := \pi(\alpha,\beta) $$
\noindent
for all 1-forms $\alpha, \beta$ on $P$; by abuse of notation, let $\pi^{\#}$ also denote the induced map on all forms on $P$. The distribution $\mathrm{im} \; \pi^{\#}$ is called the \emph{characteristic distribution} of $\pi$.

\begin{defn}[\cite{severa_weinstein}]\label{defn:twisted_poisson}
  Let $P$ be a smooth manifold. A pair $(\pi, \phi)$, where $\pi$ is a bivector field and $\phi$ a closed $3$-form, is a \emph{twisted Poisson structure} if
  \begin{equation}
    \label{eq:1}
    \llbracket \pi,\pi \rrbracket=\frac{1}{2}\wedge^3 \pi^{\#}\phi,
  \end{equation}
  \noindent
  where $\llbracket .,. \rrbracket $ denotes the Schouten-Nijenhuis bracket. In this case, say that $\pi$ is a $\phi$-\emph{twisted} Poisson structure on $P$, that $\phi$ is a \emph{twisting form} of $\pi$, and that the triple $(P,\pi,\phi)$ is a \emph{twisted Poisson} manifold.
\end{defn}

\begin{obs}\label{rk:twisted_poisson}
  If $(\pi, \phi)$ is a twisted Poisson structure on $P$ and $\{.,.\}$ denotes the associated bracket on $\mathrm{C}^{\infty}(P)$, equation \eqref{eq:1} becomes 
  $$  \{\{h_1,h_2\},h_3\} + \text{ c.p. } = \phi(X_{h_1},X_{h_2},X_{h_3}), $$
  \noindent 
  where, for each $j$, $X_{h_j}$ denotes the Hamiltonian vector field of $h_j$, \textit{i.e.} $X_{h_j} = \pi^{\#}(\de h_j)$, and c.p. stands for cyclic permutation. In other words, the 3-form $\phi$ controls the failure of $\{.,.\}$ to satisfy the Jacobi identity. As in the context of Poisson manifolds, functions whose Hamiltonian vector fields vanish identically are called \emph{Casimirs} of $\tpm$.
\end{obs}

As stated in the Introduction, symplectic manifolds are to Poisson geometry as almost symplectic manifolds are to twisted Poisson geometry.

\begin{defn}\label{defn:almost_symplectic_manifolds}
  An \emph{almost symplectic} form on smooth manifold $M$ is a non-degenerate 2-form $\sigma$ on $M$. An \emph{almost symplectic manifold} is a pair $\asm$ consisting of a smooth manifold $M$ and an almost symplectic form $\sigma$ on $M$.
\end{defn}

\begin{obs}\label{rk:almost_symplectic_almost_complex}
  Let $\asm$ be an almost symplectic manifold and let $\pi_{\sigma}$ be bivector field induced by the almost Poisson bracket on $\mathrm{C}^{\infty}(M)$ defined by $\sigma$. By definition, $(\pi_{\sigma}, \de \sigma)$ is a twisted Poisson structure on $M$.
\end{obs}

A twisted Poisson structure $(\pi,\phi)$ on $P$ defines a \emph{Lie algebroid} structure on $\cotan P \to P$ with anchor map $\pi^{\#}: \cotan P \to \mathrm{T} P$ and Lie bracket defined by 
\begin{equation}
  \label{eq:2}
  [\alpha, \beta]_{\phi} := L_{\pi^{\#}\alpha}\beta - L_{\pi^{\#}\beta}\alpha - \de \pi(\alpha,\beta)+ \phi(\pi^{\#}\alpha,\pi^{\#}\beta,-),
\end{equation}
\noindent
for $\alpha, \beta$ 1-forms on $P$ (cf. \cite{cattaneo_xu,severa_weinstein}). Unlike the case of Poisson structures, Hamiltonian vector fields do not form a Lie subalgebra of all vector fields, since the above equation implies that, for all $h_1,h_2 \in \mathrm{C}^{\infty}(P)$, 
\begin{equation}
  \label{eq:3}
  [X_{h_1},X_{h_2}] = X_{\{h_1,h_2\}} + \pi^{\#}(\phi(X_{h_1},X_{h_2},-)).
\end{equation}
\noindent
Since the (possibly singular) distribution defined by the image of the anchor of a Lie algebroid is integrable in the sense of Frobenius-Stefan-Sussmann, it follows that the distribution $\mathrm{im} \; \pi^{\#}$ is integrable with (possibly singular) foliation $\mathcal{F}$. Let $O$ be a leaf of $\mathcal{F}$ and define a smooth 2-form $\tau_O \in \Omega^2(O)$ by 
$$ \tau_O(\pi^{\#}(\alpha),\pi^{\#}(\beta)) : = \pi(\alpha,\beta), $$
\noindent
where $\alpha, \beta$ are 1-forms on $P$ (recall that, for all $x \in O$, $\mathrm{T}_xO = \pi^{\#}_x (\mathrm{T}^*_x P)$ by definition). If $\iota :O \hookrightarrow P$ denotes the inclusion, then
$$ \iota^*\phi = \de \tau_O. $$
\noindent
Note that the 2-form $\tau_O$ defined above is, by construction, non-degenerate. Therefore, the leaves of $\mathcal{F}$ inherit almost symplectic forms from $\pi$ and $\mathcal{F}$ is called the \emph{almost symplectic foliation} associated to $\tpm$. 

\begin{remark}\label{rk:phi_not_unique}
  \mbox{}
  \begin{enumerate}[label=\arabic*), ref=\arabic*)]
  \item \label{item:26}   As noted in \cite{balseiro_garcia_naranjo} in the more general context of twisted Dirac structures, if $\pi$ is a bivector field on a manifold $P$ admitting a twisting form $\phi$ and if $\mathcal{F}$ denotes the almost symplectic foliation of $\tpm$, then $\phi'$ is another twisting form for $\pi$ if and only if $\phi' = \phi +\theta$, where $\theta$ is a closed 3-form on $P$ vanishing on the leaves of $\mathcal{F}$. For instance, any Poisson bivector $\pi$ on $P$ defines a $\phi$-twisted Poisson structure, where $\phi$ is any closed 3-form on $P$ vanishing on the leaves of the symplectic foliation of $(P,\pi)$;
  \item \label{item:27}  Fix a twisted Poisson manifold $\tpm$ and let $\phi'$ be another twisting form for $\pi$. Note that equation \eqref{eq:3} holds with $\phi'$ replacing $\phi$, since $\phi'-\phi$ vanishes along $\mathcal{F}$; on the other hand, for all 1-forms $\alpha, \beta$ on $P$, 
  $$ [\alpha, \beta]_{\phi'} - [\alpha, \beta]_{\phi} = (\phi'-\phi)(\pi^{\#}\alpha,\pi^{\#}\beta,-), $$
  \noindent
  where $[.,.]_{\phi'}, [.,.]_{\phi}$ are Lie brackets on $\cotan P \to P$ defined as in equation \eqref{eq:2}. In particular, the choice of twisting form for the bivector $\pi$ determines the Lie algebroid structure on $\cotan P \to P$; two choices of twisting forms may induce Lie algebroids which are distinct in the sense that one of them is integrable in the sense of \cite{crainic_fernandes_lie}, while the other is not (cf. Remark \ref{rk:complete_realisations_integrability}).
  \end{enumerate}
\end{remark}

In this note, the main focus is on twisted Poisson manifolds $\tpm$ whose characteristic distribution $\mathrm{im} \; \pi^{\#}$ has constant rank (cf. Proposition \ref{prop:isotropic_implies_regular}).

\begin{defn}\label{defn:regular_twisted_Poisson}
  A twisted Poisson manifold $\tpm$ is said to be \emph{regular} if its characteristic distribution $\mathrm{im} \; \pi^{\#}$ is of constant rank.
\end{defn}

\begin{example}\label{ex:almost_symplectic_twisted_poisson}
  Let $\pi$ be a bivector field on a manifold $P$ whose characteristic distribution $\mathrm{im} \; \pi^{\#}$ is integrable and of constant rank. Then there exists an exact 3-form $\de \tau \in \Omega^3(P)$ such that $(\pi,\de \tau)$ is a twisted Poisson structure on $P$ (cf. \cite{balseiro_garcia_naranjo}).
\end{example}

Just as in the case of regular Poisson manifolds, a regular twisted Poisson manifold $(P,\pi,\phi)$ has extra structure. For the purposes of this paper, the following properties of regular twisted Poisson manifolds are relevant:

\begin{enumerate}[label=(P\arabic*), ref=(P\arabic*)]
\item \label{item:11} There exists a globally defined smooth 2-form $\tau$ on $P$ which restricts to the almost symplectic form induced by $\pi$ on the leaves of $\mathcal{F}$ (cf. \cite{balseiro_garcia_naranjo} and Example \ref{ex:almost_symplectic_twisted_poisson} above). Note, however, that $\tau$ is not necessarily non-degenerate (for instance, $P$ may be odd dimensional), nor is it unique; in fact, $\tau$ is unique up to addition of a 2-form which vanishes on the leaves of $\mathcal{F}$;
\item \label{item:12} Since $\tpm$ is regular, the kernel of $\pi^{\#}: \cotan P \to \mathrm{T}P$ is a smooth subbundle of $\cotan P \to P$. By definition, for each $x \in P$, an element of $\ker \pi^{\#}(x)$ annihilates all vectors in $\mathrm{im} \;\pi^{\#}(x)$; thus, if $\nu^*_x \mathcal{F}$ denotes the conormal space to the almost symplectic foliation $\mathcal{F}$, then $\ker \pi^{\#}(x) \subset \nu^*_x \mathcal{F}$, but they are of equal dimension as vector spaces, so that $\ker \pi^{\#}(x) = \nu^*_x \mathcal{F}$. Locally, $\nu^*_x \mathcal{F}$ is spanned by the differentials of Casimir functions. Let $\phi'$ be any twisting form for $\pi$; the Lie algebroid structure defined by the twisted Poisson structure $(\pi,\phi')$ on $\cotan P \to P$ restricts to define a Lie algebroid structure on $\nu^* \mathcal{F} \to P$ with zero anchor and trivial Lie bracket (in fact, $\nu^*\mathcal{F} \to P$ is a bundle of abelian Lie algebras). Observe that this fact is intrinsic to the bivector $\pi$ and does not depend on the choice of twisting form $\phi$;
\item \label{item:4} Let $\tau \in \Omega^2(P)$ be a 2-form as in Property \ref{item:11}. Note that the bivector $\pi$ is also $\de \tau$-twisted, which implies that $\de \tau - \phi$ is a closed 3-form which vanishes on the leaves of $\mathcal{F}$. Let $\Omega^m_{\mathrm{rel}}(P;\mathcal{F})$ denote the vector space of differential $m$-forms on $P$ which vanish on the leaves of $\mathcal{F}$. Differential forms vanishing on the leaves of $\mathcal{F}$ form a subcomplex of the standard de Rham complex; thus define the \emph{relative} de Rham cohomology groups 
  $$\mathrm{H}^m_{\mathrm{rel}}(P;\mathcal{F}) = \frac{\ker \de :\Omega^m_{\mathrm{rel}}(P;\mathcal{F}) \to \Omega^{m+1}_{\mathrm{rel}}(P;\mathcal{F})}{\mathrm{im} \;\,\de :\Omega^{m-1}_{\mathrm{rel}}(P;\mathcal{F}) \to \Omega^{m}_{\mathrm{rel}}(P;\mathcal{F})} $$
  \noindent
  for all $m \in \mathbb{N}$ (cf. \cite{vaisman}). The form $\de \tau - \phi$ defines a cohomology class $\xi_{\phi}=[\de \tau - \phi] \in \mathrm{H}^3_{\mathrm{can}}(P;\mathcal{F})$. If $\tau' \in \Omega^2(P)$ is another 2-form constructed as in \ref{item:11}, then $\tau - \tau'$ vanishes on the leaves of $\mathcal{F}$, which implies that 
  $$ \de \tau' - \phi = \de \tau + \de (\tau' - \tau) - \phi. $$
  \noindent
  Therefore $[\de \tau' - \phi] = [\de \tau - \phi]$, which implies that $\xi_{\phi}$ is independent of the choice of 2-form $\tau$ constructed in \ref{item:11}. 
\end{enumerate}

\begin{defn}\label{defn:characteristic_class}
  Let $\tpm$ be a regular twisted Poisson manifold. The cohomology class $\xi_{\phi} \in \mathrm{H}^3_{\mathrm{can}}(P;\mathcal{F})$ is called the \emph{characteristic class} of $\tpm$.
\end{defn}

Note that if $\phi = 0$, the characteristic class $\xi_{\phi}$ corresponds to the characteristic class of the regular Poisson manifold $(P,\pi)$ as defined, for instance, in \cite{vaisman}. In general, $\xi_{\phi}$ depends upon the choice of twisting form $\phi$.

\section{Mechanical model}\label{sec:mechanical-model}
\subsection{Integrable almost-symplectic Hamiltonian systems}\label{sec:integr-almost-sympl}
In this section, the mechanical model associated to the geometric objects studied in Section \ref{sec:twist-isotr-real} is introduced, following \cite{fasso_sansonetto}. As mentioned in the Introduction, the systems described in \cite{fasso_sansonetto} lie in between non-commutative integrable Hamiltonian systems (cf. \cite{fasso,mishchenko_fomenko,nek}) and nonholonomic systems (cf. \cite{balseiro_garcia_naranjo,bates_cushman}). Most results are recalled below without proof (cf. \cite{fasso_sansonetto}); some are proved using different, more intrinsic techniques in order to introduce some notation and ideas used in Section \ref{sec:twist-isotr-real}. \\

Let $\asm$ be an almost symplectic manifold with induced almost-Poisson bracket $\{.,.\}_{\sigma}$, and, for $f \in \mathrm{C}^{\infty}(M)$, let $X_{f}$ denote the Hamiltonian vector field of $f$, namely the unique vector field satisfying 
$$ i(X_f)\sigma = \de f, $$
\noindent
where $i$ denotes the interior product. 

\begin{defn}\label{defn:strongly_hamiltonian}
  A Hamiltonian vector field $X_f$ on $\asm$ is said to be \emph{strongly Hamiltonian} if $L_{X_f} \sigma = 0$, where $L$ denotes the Lie derivative.
\end{defn}

\begin{obs}\label{rk:equivalent_cartan}
  \mbox{}
  \begin{enumerate}[label=\arabic*), ref=\arabic*]
  \item \label{item:30} In light of Cartan's formula, requiring that a Hamiltonian vector field $X_f$ is strongly Hamiltonian is equivalent to asking that $i(X_f) \de \sigma = 0$;
  \item \label{item:31} Denote by $\mathrm{C}^{\infty}_{\mathrm{str}}(M)$ the vector subspace of $\mathrm{C}^{\infty}(M)$ consisting of functions whose Hamiltonian vector fields are strongly Hamiltonian. The restriction of $\{.,.\}_{\sigma}$ to $\mathrm{C}^{\infty}_{\mathrm{str}}(M)$ is closed; in fact, $(\mathrm{C}^{\infty}_{\mathrm{str}}(M),\{.,.\}_{\sigma})$ is a Lie algebra.
  \end{enumerate}
\end{obs}

The notion of strongly Hamiltonian vector fields is central to the definition of non-commutative Hamiltonian integrability in the almost symplectic category. In what follows, the notion of integrability given by Definition \ref{defn:ncias} generalises the work of Nehoro\v sev (cf. \cite{nek}), but it can easily be adapted to generalise integrability in the sense of Mi\v s\v cenko and Fomenko (cf. \cite{mishchenko_fomenko}). In fact, all results stated and proved below hold for both types of generalisations. 

\begin{defn}\label{defn:ncias}
  Let $\asm$ be a $2n$-dimensional almost symplectic manifold. A \emph{non-commutative integrable Hamiltonian system} without singularities (NCIHS for short) on $\asm$ is a map
  $$ F = (f_1,\ldots,f_{2n-k}) : \asm \to \R^{2n-k}, $$
  \noindent
  where $1\leq k \leq n$, which satisfies the following properties:
  \begin{enumerate}[label=(NC\arabic*), ref=(NC\arabic*)]
  \item \label{item:32} for all $j=1,\ldots,k$, $f_j \in \mathrm{C}^{\infty}_{\mathrm{str}}(M)$;
  \item \label{item:33} for all $j=1,\ldots,k$, $l=1,\ldots, 2n-k$, $\{f_j,f_l\}_{\sigma}=0$;
  \item \label{item:34} $F$ is a submersion, \textit{i.e.} $\de f_1 \wedge \ldots \wedge \de f_{2n-k} \neq 0$;
  \item \label{item:37} The fibres of $F$ are compact and connected.
  \end{enumerate}
\end{defn}

\begin{remark}\label{rk:ncias}
  \mbox{}
  \begin{enumerate}[label=\arabic*), ref=\arabic*]
  \item \label{item:35} When $\asm$ is a symplectic manifold, condition \ref{item:32} is automatically satisfied; in this case, Definition \ref{defn:ncias} recovers the notion of non-commutative integrable Hamiltonian systems without singularities with compact and connected fibres in the sense of \cite{nek} (cf. \cite{fasso} for an exhaustive introduction to the local theory of these systems);
  \item \label{item:6} Conditions \ref{item:34} and \ref{item:37} can be weakened. Doing so to the former allows for singular points, \textit{i.e.} points on $M$ where the integrals $f_1,\ldots, f_{2n-k}$ are not functionally independent. Generally this is done by controlling how `large' the singular set is, \textit{e.g.} the subset of $M$ consisting of regular points of $F$ may be stipulated to be dense. Condition \ref{item:37} can be relaxed by replacing the assumption on the fibres being compact with them being \emph{complete}, \textit{i.e.}  for $i=1,\ldots,k$, the Hamiltonian vector fields $X_{f_i}$ are complete. On the other hand, the condition on connectedness is very hard to dispense with in general (cf. \cite{pelayo_ratiu_vu_ngoc});
  \item \label{item:36} Condition \ref{item:32} is very strong and, in general, is not satisfied by nonholonomic systems. However, geometrically, this is not a novel hypothesis, since, as mentioned in \cite{fasso_sansonetto}, it is encountered in the context of reduction of almost-Dirac manifolds (cf. \cite{blankenstein_ratiu}). Moreover, as non-commutative integrable Hamiltonian systems are a local model for isotropic realisations of Poisson manifolds (cf. \cite{daz_delz}), these systems are local models for twisted isotropic realisations of twisted Poisson manifolds (cf. Definition \ref{defn:symplectic_realisations_almost_poisson}).
  \end{enumerate}
\end{remark}

For NCIHS on almost symplectic manifolds as defined above, a version of the Liouville-Arnol'd theorem can be proved.

\begin{theorem}[\cite{fasso_sansonetto}]\label{thm:lma_almost}
  Let $F : \asm \to \R^{2n-k}$ be a NCIHS on an almost symplectic manifold $\asm$. Then
  \begin{enumerate}[label=\roman*), ref=\roman*]
  \item \label{item:39} The fibres of $F$ are diffeomorphic to $\mathbb{T}^k$ and can be endowed with smoothly varying flat, torsion-free, complete connections;
  \item \label{item:41} There exist local action-angle variables, \textit{i.e.} for all $x \in F(M)$, there exists an open neighbourhood $U \subset F(M)$ containing $x$ and a diffeomorphism 
    $$ (\mathbf{a},\mathbf{b},\boldsymbol{\alpha}) : F^{-1}(U) \to V \times W \times \mathbb{T}^k, $$
    \noindent
    where $V \subset \R^k, W \subset \R^{2n-2k}$ are open sets, such that $\sigma|_{F^{-1}(U)}$ takes the form
    $$ \sum\limits_{l=1}^k \de a^l \wedge \de \alpha^l + \frac{1}{2}\sum\limits_{l,r=1}^k A_{lr} \de a^l \wedge \de a^r + \sum\limits_{l=1}^k\sum\limits_{u=1}^{2n-2k}B_{lu} \de a^l \wedge \de b^u + \frac{1}{2} \sum\limits_{u,v=1}^{2n-2k} C_{uv} \de b^u \wedge \de b^v, $$
    \noindent
    the matrices $A = A^{-T}, B, C=C^{-T}$ depend smoothly on $(a,b)$ and, if $k \neq n$, $C$ is everywhere non-singular. Coordinates $\mathbf{a}$ are called \emph{actions}, while $\boldsymbol{\alpha}$ are known as \emph{angles}.
  \end{enumerate}
\end{theorem}

\subsection{Relation to twisted Poisson manifolds}\label{sec:relat-twist-poiss}
Associated to a NCIHS $F: \asm \to \R^{2n-k}$ on $\asm$ is a commutative diagram
\begin{equation*}
    \xymatrix{ \asm \ar[r]^-{F} \ar[dr]_-{F'=(f_1,\ldots,f_k)} & \R^{2n-k} \ar[d]^-{\mathrm{pr}_{\R^k}} \\
    & \R^k,}
\end{equation*}
\noindent
where $\mathrm{pr}_{\R^k}: \R^{2n-k} \to \R^k$ denotes projection onto the first $k$ components. Properties \ref{item:33} and \ref{item:34} imply that, for all $m \in M$, the Hamiltonian vector fields $X_{f_1}(m),\ldots,X_{f_k}(m)$ span $\ker D_mF$, while $X_{f_1}(m),\ldots,X_{f_{2n-k}}(m)$ span $\ker D_{m}F'$. Moreover the regular foliations defined by the integrable distributions 
\begin{equation*}
  D_1 = \mathrm{span} \langle X_{f_1},\ldots, X_{f_k} \rangle \qquad \text{and} \qquad D_2 = \mathrm{span} \langle X_{f_1},\ldots,X_{f_{2n-k}} \rangle
\end{equation*}
\noindent
are $\sigma$-orthogonal to one another. Observe that the fact that $\sigma$ is non-degenerate is very important for these assertions. Noticing that $\sigma$ vanishes when restricted to the fibres of $F$ by property \ref{item:33}, it follows that these fibres are \emph{isotropic}, while the fibres of $F'$ are \emph{coisotropic}, where the terms isotropic/coisotropic have the same meaning as in the symplectic context. Thus a NCIHS endows $\asm$ with an isotropic/coisotropic pair of foliations just as in the symplectic case (cf. \cite{daz_delz,fasso_sansonetto,fasso}). \\

Henceforth, assume the following unless otherwise stated:

\begin{enumerate}[label=(NC5), ref=(NC5)]
\item \label{item:40} $P = F(M)$ is a smooth manifold.
\end{enumerate}

Under hypotheses \ref{item:32} -- \ref{item:40}, a NCIHS $F: \asm \to P$ induces an almost Poisson structure on $P$, which makes $F$ into an almost Poisson morphism (cf. \cite{fasso_sansonetto} for a different proof using action-angle variables). 

\begin{lemma}\label{lemma:quasi_poisson_base}
  Let $F: \asm \to P$ be a NCIHS. There exists a unique almost Poisson structure $\{.,.\}_P$ on $P$ which makes $F$ into an almost Poisson morphism, \textit{i.e.} for all $h_1, h_2 \in \mathrm{C}^{\infty}(P)$,
  $$ \{F^* h_1, F^*h_2 \}_{\sigma} = F^* \{h_1,h_2\}_P. $$
\end{lemma}
\begin{proof}
  It suffices to show that if $h_1, h_2 \in \mathrm{C}^{\infty}(P)$, then $\{F^*h_1,F^*h_2\}_{\sigma}$ is a basic function and, thus, it is equal to $F^* g_{12}$, for some $g_{12} \in \mathrm{C}^{\infty}(P)$. If this holds, then, for $h_1,h_2 \in \mathrm{C}^{\infty}(P)$, set
  \begin{equation}
    \label{eq:31}
    \{h_1,h_2\}_P :=g_{12};
  \end{equation}
  \noindent
  skew-symmetry, bilinearity and the Leibniz identity follow from the fact that $\{.,.\}_{\sigma}$ is an almost Poisson bracket. Note that if equation \eqref{eq:31} defines an almost Poisson bracket on $P$, it defines the unique such bracket making $F$ into an almost Poisson morphism. \\

Let $V \subset \mathrm{T} M \to M$ be the subbundle of `vertical' vector fields, \textit{i.e.} for all $m \in M$, $V_m =\ker D_mF$. Let $Y \in \Gamma(V)$ be a vertical vector field, \textit{i.e.} a smooth section of $V \to M$. The result is proved if, for all $h_1,h_2 \in \mathrm{C}^{\infty}(P)$, 
\begin{equation*}
  Y \{F^*h_1,F^*h_2\}_{\sigma} = 0.
\end{equation*}
\noindent
Fix such a vector field $Y$. Note that the Hamiltonian vector fields $X_{f_1},\ldots,X_{f_k}$ are vertical and, for all $m \in M$, form a frame of $V_m$ by properties \ref{item:33} and \ref{item:34}. In particular, for $j=1,\ldots,k$, there exist smooth functions $P_j:M \to \R$ such that, for all $m \in M$,
$$ Y(m) = \sum\limits_{j=1}^k P_j(m) X_{f_j}(m). $$
\noindent
Therefore, the result follows if it is shown that, for all $j=1,\ldots,k$, 
\begin{equation}
  \label{eq:35}
  X_{f_j} \{F^*h_1,F^*h_2\}_{\sigma} = 0.
\end{equation}
\noindent
Fix $j$ and note that
\begin{equation}
  \label{eq:32}
    X_{f_j} \{F^*h_1,F^*h_2\}_{\sigma} =  \sigma(X_{\{F^*h_1,F^*h_2\}_{\sigma}},X_{f_j}) =\{\{F^*h_1,F^*h_2\}_{\sigma},f_j\}_{\sigma}.
\end{equation}
\noindent
However, by definition of $\{.,.\}_{\sigma}$, 
\begin{equation}
  \label{eq:33}
  \{\{F^*h_1,F^*h_2\}_{\sigma},f_j\}_{\sigma} + \text{ c.p. } = \de \sigma (X_{F^*h_1},X_{F^*h_2},X_{f_j}).
\end{equation}
\noindent
The right hand side of equation \eqref{eq:33} vanishes, as the vector field $X_{f_j}$ is strongly Hamiltonian by property \ref{item:32}. Therefore, 
\begin{equation}
  \label{eq:34}
  \begin{split}
    X_{f_j} \{F^*h_1,F^*h_2\}_{\sigma} & = \{\{F^*h_1,F^*h_2\}_{\sigma},f_j\}_{\sigma} \\
  & = \{F^*h_1,\{F^*h_2,f_j\}_{\sigma}\}_{\sigma} + \{F^*h_2,\{f_j,F^*h_1\}_{\sigma}\}_{\sigma}. 
  \end{split}
\end{equation}
\noindent
Note that for $l=1,2$, 
$$ \{F^*h_l,f_j\}_{\sigma} = - X_{f_j}(F^*h_l) = 0, $$
\noindent
since $X_{f_j} \in \Gamma(V)$. Therefore, the right hand side of equation \eqref{eq:34} vanishes, proving that equation \eqref{eq:35} holds. Since $j$ is arbitrary, this completes the proof of the lemma.
\end{proof}

In fact, more structure on the almost Poisson structure $\{.,.\}_P$ can be inferred from the underlying NCIHS $F: \asm \to P$. Since 
$$ \ker D_{m}F = \mathrm{span} \langle X_{f_1}(m),\ldots,X_{f_k}(m) \rangle $$
\noindent
for all $m \in M$ and $X_{f_1},\ldots,X_{f_k}$ are strongly Hamiltonian by property \ref{item:32}, it follows that the 3-form $\de \sigma$ is \emph{basic}, \textit{i.e.} there exists a 3-form $\phi \in \Omega^3(P)$ such that
\begin{equation}
  \label{eq:36}
  \de \sigma = F^* \phi.
\end{equation}
\noindent
Note that equation \eqref{eq:36} implies that $\de \phi = 0$, since $\de \sigma$ is exact and $F$ is a surjective submersion. As $\{.,.\}_{\sigma}$ is a $\de \sigma$-twisted Poisson structure, it is natural to ask whether $\{.,.\}_P$ is $\phi$-twisted; this is the content of the next lemma, which, to the best of our knowledge, appears here for the first time. 

\begin{lemma}\label{lemma:twisted_base}
  The bracket $\{.,.\}_P$ is $\phi$-twisted Poisson.
\end{lemma}
\begin{proof}
  The result is proved if, for all $h_1,h_2,h_3 \in \mathrm{C}^{\infty}(P)$, the following equality holds
  \begin{equation}
    \label{eq:39}
    \{\{h_1,h_2\}_P,h_3\}_P + \text{ c.p. } = \phi(X_{h_1},X_{h_2},X_{h_3}),
  \end{equation}
  \noindent
  where, for $j=1,2,3$, $X_{h_j}$ denotes the Hamiltonian vector field of $h_j$ with respect to the almost Poisson bracket $\{.,.\}_P$. Fix $h_1,h_2,h_3$ as above and note that, by definition,
  \begin{equation}
    \label{eq:37}
    \{\{F^*h_1,F^*h_2\}_{\sigma},F^*h_3\}_{\sigma} + \text{ c.p. } = \de \sigma (X_{F^*h_1},X_{F^*h_2},X_{F^*h_3}).
  \end{equation}
  \noindent
  Recall that $\de \sigma = F^* \phi$, so that, for all $m\in M$, the right hand side of equation \eqref{eq:37} becomes
  \begin{equation*}
    \begin{split}
      \de \sigma_m (X_{F^*h_1}(m),X_{F^*h_2}(m) &,X_{F^*h_3}(m)) = (F^* \phi)_m (X_{F^*h_1}(m),X_{F^*h_2}(m),X_{F^*h_3}(m)) \\
      &=\phi_{F(m)}(D_mF (X_{F^*h_1}), D_mF(X_{F^*h_2}),D_mF (X_{F^*h_3})).
    \end{split}
  \end{equation*}
  \noindent
  Just as for Poisson morphisms, the fact that $F$ is an almost Poisson morphism implies that for all $j=1,2,3$ and for $m \in M$,
  $$ D_m F(X_{F^*h_j}) = X_{h_j} (F(m)) $$
  \noindent
  (cf. \cite{marsden_ratiu}). In particular, it follows that
  \begin{equation*}
    \phi_{F(m)}(D_mF (X_{F^*h_1}), D_mF(X_{F^*h_2}),D_mF (X_{F^*h_3})) = F^*(\phi(X_{h_1},X_{h_2},X_{h_3}))(m).
  \end{equation*}
  \noindent
  Therefore the right hand side of equation \eqref{eq:37} equals $F^*(\phi(X_{h_1},X_{h_2},X_{h_3}))$. On the other hand, the surjective submersion $F: (M,\{.,.\}_{\sigma}) \to (P,\{.,.\}_P)$ is an almost Poisson morphism, so that, for all $j,l$, 
  $$ \{F^*h_j,F^*h_l\}_{\sigma} = F^*\{h_j,h_l\}_P; $$
  \noindent
  applying the above identity several times to the left hand side of equation \eqref{eq:37}, obtain that
  \begin{equation*}
    \{\{F^*h_1,F^*h_2\}_{\sigma},F^*h_3\}_{\sigma} + \text{ c.p. } = F^*(\{\{h_1,h_2\}_P,h_3\}_P + \text{ c.p.}).
  \end{equation*}
  \noindent
  Thus,
  \begin{equation*}
    F^*(\{\{h_1,h_2\}_P,h_3\}_P + \text{ c.p.}) = F^*(\phi(X_{h_1},X_{h_2},X_{h_3}));
  \end{equation*}
  \noindent
  since $F$ is a surjective submersion, it follows that equation \eqref{eq:39} holds. This completes the proof of the lemma.
\end{proof}

\begin{remark}\label{rk:twisted_base_coordinates}
  An alternative, equivalent proof of the above result can be obtained by using action-angle variables as in Proposition 6 of \cite{fasso_sansonetto} and by noticing that condition \ref{item:32} implies that $\de \sigma$ is a basic form.
\end{remark}

\section{Twisted isotropic realisations}\label{sec:twist-isotr-real}
In this section, twisted isotropic realisations (with compact and connected fibres) of twisted Poisson manifolds are introduced (cf. Section \ref{sec:defin-basic-prop}) and their topological (in fact, smooth) invariants are constructed, thus yielding a smooth classification of these objects (cf. Section \ref{sec:smooth-constructions}).

\subsection{Definition and basic properties}\label{sec:defin-basic-prop}
\begin{defn}[\cite{bursztyn_crainic_weinstein_zhu}]\label{defn:symplectic_realisations_almost_poisson}
  A \emph{twisted isotropic realisation} (with compact fibres) of a twisted Poisson manifold $\tpm$ is a surjective submersion $p: \asm \to \tpm$ from an almost symplectic manifold onto $P$ which satisfies the following properties:
  \begin{enumerate}[label=(IR\arabic*), ref=(IR\arabic*)]
  \item \label{item:13} $p$ is an almost Poisson morphism, \textit{i.e.} for all $h_1, h_2 \in \mathrm{C}^{\infty}(P)$,
    $$ \{p^*h_1,p^*h_2\}_{\sigma} = p^*\{h_1,h_2\}, $$
    \noindent
    where $\{.,.\}_{\sigma}, \{.,.\}$ are the brackets on $\mathrm{C}^{\infty}(M), \mathrm{C}^{\infty}(P)$ induced by $\sigma$ and $\pi$ respectively;
  \item \label{item:42} $ p^*\phi = \de \sigma; $
  \item \label{item:5} for all $x \in P$, $p^{-1}(x)$ is an isotropic submanifold of $\asm$, \textit{i.e.} $\sigma|_{p^{-1}(x)} \equiv 0$;
  \item \label{item:9} The fibres of $F$ are compact and connected.
  \end{enumerate}
\end{defn}

\begin{obs}\label{obs:relation_ncihs_twisted_isotropic}
  Lemmata \ref{lemma:quasi_poisson_base} and \ref{lemma:twisted_base} of Section \ref{sec:relat-twist-poiss} prove how properties \ref{item:32} -- \ref{item:40} of a NCIHS $F: \asm \to P \subset \R^{2n-k}$ imply that $F : \asm \to \tpm$ is a twisted isotropic realisation, where $\pi$ is the bivector field associated to the almost Poisson bracket $\{.,.\}_P$ on $P$ constructed in Lemma \ref{lemma:quasi_poisson_base}. This generalises the correspondence between NCIHS on symplectic manifolds and isotropic realisations of Poisson manifolds to the almost symplectic category.
\end{obs}

\begin{remark}\label{example:poisson}
  \mbox{}
  \begin{enumerate}[label=\arabic*), ref=\arabic*]
  \item \label{item:2} Given a twisted symplectic realisation $p: \asm \to \tpm$, suppose that $\asm$ is a symplectic manifold. Condition \ref{item:42} implies that $p^*\phi = 0$; since $p$ is a surjective submersion, $\phi = 0$, which implies that the bivector field $\pi$ defines an honest Poisson structure on $P$. This recovers the notion of isotropic realisations of Poisson manifolds (cf. \cite{daz_delz});
  \item \label{item:1} In general, property \ref{item:9} is not included in the definition of (twisted) isotropic realisations of (twisted) Poisson manifolds. However, most classification results in the literature on isotropic realisations of Poisson manifolds assume connectedness and completeness of the fibres (cf. Remark \ref{rk:ncias}.\ref{item:6} for an analogous consideration regarding NCIHS).
  \end{enumerate}
\end{remark}

\begin{example}\label{exm:regular_twisted_poisson_admit_wtir}
  Let $\pi$ be a bivector on $P$ whose characteristic foliation $\mathcal{F}$ is integrable in the sense of Frobenius. Example \ref{ex:almost_symplectic_twisted_poisson} guarantees the existence of a globally defined 2-form $\tau \in \Omega^2(P)$ such that $\de \tau$ is a twisting form for $\pi$. Let $\mathrm{pr}:\nu^* \mathcal{F} \to P$ denote the conormal bundle to $\mathcal{F}$ and let $\jmath : \nu^* \mathcal{F} \hookrightarrow \cotan P$ denote the natural inclusion. Consider the 2-form
  $$ \sigma_0 = \jmath^* \Omega_{\mathrm{can}} + \mathrm{pr}^* \tau, $$
  \noindent
  where $\Omega_{\mathrm{can}}$ denotes the canonical symplectic form on $\cotan P$. Then
  $$\mathrm{pr} : (\nu^* \mathcal{F},\sigma_0) \to (P,\pi,\de \tau) $$
  \noindent
  is a twisted isotropic realisation with non-compact fibres.
\end{example}

The following proposition proves that only regular twisted Poisson manifolds admit twisted isotropic realisations, generalising the analogous statement in the Poisson category (cf. Proposition 8.10 in \cite{vaisman}). Henceforth, given a twisted isotropic realisation $\ir$, set $\dim M = 2n$ and $\dim P = 2n -k$ unless otherwise stated.

\begin{proposition}\label{prop:isotropic_implies_regular}
  Let $\ir$ be a twisted isotropic realisation of a twisted Poisson manifold $\tpm$. Then $\tpm$ is necessarily regular with corank $k$.
\end{proposition}
\begin{proof}
  It suffices to show that for all $x_0 \in P$, $\mathrm{rk} \; \pi^{\#}(x_0) = 2n-2k$. To this end, it suffices to prove that if $x^1,\ldots,x^{2n-k}$ are local coordinates near $x_0 \in P$, the Hamiltonian vector fields $X_{x^1},\ldots,X_{x^{2n-k}}$ span a subspace of $\mathrm{T}_{x_0}P$ of dimension $2n-2k$. Fix $x_0 \in P$. Since $p$ is surjective, there exists $m \in M$ such that $p(m)=x_0$. Property \ref{item:13} implies that, for any $h \in \mathrm{C}^{\infty}(P)$,
  $$ (D_m p) X_{p^*h} = X_h (p(m)); $$
  \noindent
  thus it suffices to prove that 
  $$ (D_m p) X_{p^*x^1},\ldots, (D_m p) X_{p^* x^{2n-k}} $$
  \noindent
  span a subspace of $\mathrm{T}_{x_0}P$ of dimension $2n-2k$. Note that, for all $j$, $X_{p^*x^j}(m) \in (\ker D_mp)^{\sigma}$, since $p^*x^j$ is constant along the fibres of $p$. If
  \begin{equation}
    \label{eq:16}
    \text{span}\langle X_{p^*x^1}(m),\ldots,X_{p^*x^{2n-k}}(m) \rangle = (\ker D_mp)^{\sigma}
  \end{equation}
  \noindent
  the result follows, since
  $$ (D_m p)\;(\text{span} \langle X_{p^*x^1}(m),\ldots,X_{p^*x^{2n-k}}(m) \rangle) =  \pi^{\#}(x_0)(\cotan_{x_0} P) $$
  \noindent
  and $\ker D_mp \subset (\ker D_mp)^{\sigma}$ by property \ref{item:5}. Thus it remains to show that equation \eqref{eq:16} holds and, in order to achieve this, it suffices to prove that 
  $$X_{p^*x^1}(m),\ldots,X_{p^*x^{2n-k}}(m)$$
  \noindent
are linearly independent. Suppose that, for each $j=1,\ldots,2n-k$, there exists $\lambda_j \in \R$, such that
  $$ \sum\limits_{j=1}^{2n-k}\lambda_j X_{p^* x^j}(m) = 0, $$
  \noindent
  then, by definition,
  $$ 0 = i\Bigg(\sum\limits_{j=1}^{2n-k}\lambda_j X_{p^* x^j}(m)\Bigg) \sigma = \sum\limits_{j=1}^{2n-k} \lambda_j i(X_{p^* x^j}(m)) \sigma = \sum\limits_{j=1}^{2n-k} \lambda_j (p^*\de x^j)(m). $$
  \noindent
  Since $p$ is a submersion, the above equality implies that
  $$ \sum\limits_{j=1}^{2n-k} \lambda_j \de x^j(x_0) = 0, $$
  \noindent
  which can hold if and only if $\lambda_j=0$ for all $j$, since $x^1,\ldots,x^{2n-k}$ define local coordinates near $x_0$. This concludes the proof of the proposition.
\end{proof}

\begin{remark}\label{rk:ncihs_properties}
  Let $F: \asm \to P \subset \R^{2n-k}$ be a NCIHS satisfying hypotheses \ref{item:32} -- \ref{item:40} and let $F: \asm \to \tpm$ denote the associated twisted isotropic realisation (cf. Observation \ref{obs:relation_ncihs_twisted_isotropic}). By Proposition \ref{prop:isotropic_implies_regular}, $(P,\pi,\phi)$ is regular, which implies that its almost symplectic foliation $\mathcal{F}$ is integrable in the sense of Frobenius; moreover, the rank of $\{.,.\}_P$ equals $2n-2k$. This provides an alternative proof of Proposition 6 in \cite{fasso_sansonetto}.
\end{remark}

\subsection{Smooth classification}\label{sec:smooth-constructions}
In analogy with isotropic realisations of Poisson manifolds, it is possible to classify twisted isotropic realisations of a given twisted Poisson manifold up to fibrewise diffeomorphism. In fact, the constructions in the Poisson case work almost \textit{verbatim} in the twisted Poisson context; in what follows, proofs are omitted (they can be found in \cite{crainic_fernandes_poisson,daz_delz,mackenzie,vaisman}) unless they involve new considerations depending on the structure of twisted isotropic realisations. The underlying idea is to reduce the problem to the classification of some principal groupoid bundles (cf. \cite{moerdijk_mrcun,rossi} for exhaustive reviews); this observation was explained to us by Rui Loja Fernandes. Moreover, while all proofs are intrinsic in nature, most statements are formulated also in terms of local coordinates, so as to bridge the gap between the geometrical and mechanical approaches to the subject. \\

Throughout this section, fix a twisted isotropic realisation $\ir$ and let $\pi_{\sigma}$ be the $\de \sigma$-twisted Poisson bivector on $M$ induced by $\sigma$. Let $\mathcal{F}$ denote the almost symplectic foliation of $\tpm$ and recall that $\mathrm{pr}:\nu^*\mathcal{F} \to P$ inherits the structure of a Lie algebroid with zero anchor and trivial Lie bracket (cf. Property \ref{item:12}). The idea behind this section is to prove that $\ir$ is a principal groupoid bundle for the Lie groupoid integrating $\mathrm{pr}:\nu^*\mathcal{F} \to P$.

\begin{lemma}\label{lemma:lie_algebroid_action}
  The map
  \begin{equation*}
    \begin{split}
      A: \Gamma(\nu^*\mathcal{F}) &\to \Gamma(\mathrm{T}M) \\
      \alpha &\mapsto \pi_{\sigma}^{\#}(p^*\alpha).
    \end{split}
  \end{equation*}
  \noindent
  is an infinitesimal Lie algebroid action of $\mathrm{pr}: \nu^* \mathcal{F} \to P$ along the twisted isotropic realisation $\ir$.
\end{lemma}
\begin{proof}
  The above statement is equivalent to proving the following four properties (cf. Definition 4.1.1 in \cite{mackenzie}):
  \begin{enumerate}[label=\arabic*), ref=\arabic*]
  \item \label{item:16} for all $\alpha, \beta \in \Gamma(\nu^*\mathcal{F})$, $\pi^{\#}_{\sigma}(p^*(\alpha + \beta)) = \pi^{\#}_{\sigma}(p^*\alpha) + \pi^{\#}_{\sigma}(p^*\beta)$;
  \item \label{item:17} for all $h \in \mathrm{C}^{\infty}(P)$, for all $\alpha \in \Gamma(\nu^*\mathcal{F})$, $\pi^{\#}_{\sigma}(p^*(h\alpha)) = (p^*h)\pi^{\#}_{\sigma}(p^*\alpha)$;
  \item \label{item:18} for all $\alpha \in \Gamma(\nu^*\mathcal{F})$, $\pi^{\#}_{\sigma}(p^*\alpha) \in \ker D p$;
  \item \label{item:19} for all $\alpha, \beta \in \Gamma(\nu^*\mathcal{F})$, $[\pi^{\#}_{\sigma}(p^*\alpha), \pi^{\#}_{\sigma}(p^*\beta)]=0$.
  \end{enumerate}
  Properties \ref{item:16}, \ref{item:17} and \ref{item:18} can be proved as in the case of isotropic realisations of Poisson manifolds; observe that property \ref{item:5} is crucial to the proof of \ref{item:18}. Fix $\alpha, \beta \in \Gamma(\nu^*\mathcal{F})$. Properties \ref{item:13} and \ref{item:42} imply that
  \begin{equation*}
    \begin{split}
      i([\pi^{\#}_{\sigma}(p^*\alpha), \pi^{\#}_{\sigma}(p^* \beta)])\sigma &= i(\pi^{\#}_{\sigma}(\de \pi_{\sigma}(p^*\alpha,p^*\beta)))\sigma - \de \sigma(\pi^{\#}_{\sigma}(p^*\alpha),\pi^{\#}_{\sigma}(p^* \beta), -) \\
      &= i(\pi^{\#}_{\sigma}(p^*(\de \underbrace{\pi(\alpha,\beta)}_{=0})))\sigma - (p^*\phi)(\pi^{\#}_{\sigma}(p^*\alpha),\pi^{\#}_{\sigma}(p^* \beta), -) \\
      &= - \phi (Dp (\pi^{\#}_{\sigma}(p^*\alpha)),Dp (\pi^{\#}_{\sigma}(p^* \beta)),Dp (-)) \\
      &= - \phi (\pi^{\#}(\alpha),\pi^{\#}(\beta),Dp ( -)) = 0,
    \end{split}
  \end{equation*}
  \noindent
  where the last equality uses that, for any 1-form $\gamma$ on $P$, $Dp (\pi^{\#}_{\sigma}(p^*\gamma)) = \pi^{\#}(\gamma)$, which is a consequence of property \ref{item:13}. Since $\sigma$ is non-degenerate, the above equation implies that $[\pi^{\#}_{\sigma}(p^*\alpha), \pi^{\#}_{\sigma}(p^* \beta)]=0$; as $\alpha, \beta$ are arbitrary, property \ref{item:19} is proved. This concludes the proof of the lemma.
\end{proof}

\begin{remark}\label{rk:why_the_above_proof_works}
  Non-degeneracy of $\sigma$ also implies that for each $x_0 \in P$ and for all $m \in p^{-1}(x_0)$, $A(\nu^*_{x_0} \mathcal{F}) = \ker D_m p$.
\end{remark}

\begin{remark}\label{rk:complete_realisations_integrability}
  A surjective submersion $p:\asm \to \tpm$ with connected fibres satisfying properties \ref{item:13} and \ref{item:42} is called a \emph{symplectic realisation} of $\tpm$ in \cite{bursztyn_crainic_weinstein_zhu}. Such a realisation enjoys an infinitesimal action of $\cotan P \to P$ with Lie algebroid structure defined by $(\pi,\phi)$ as in Section \ref{sec:twist-poiss-manif}; this action is defined as in Lemma \ref{lemma:lie_algebroid_action} and it maps a  1-form $\alpha$ on $P$ to $\pi^{\#}_{\sigma}(p^*\alpha)$. A symplectic realisation $\ir$ is \emph{complete} if for all compactly supported 1-forms $\alpha$ on $P$, the flow of $\pi^{\#}_{\sigma}(p^*\alpha)$ is complete. Twisted isotropic realisations in the sense of Definition \ref{defn:symplectic_realisations_almost_poisson} are complete symplectic realisations. As shown in \cite{bursztyn_crainic_weinstein_zhu,crainic_fernandes_poisson}, existence of a complete symplectic realisation of a twisted Poisson manifold $\tpm$ is equivalent to integrability of its Lie algebroid $\cotan P \to P$ with the structure defined by equation \eqref{eq:2}. As an example, let $\pi$ be a bivector on $P$ whose characteristic distribution is Frobenius integrable (so that it is necessarily of constant rank); \cite{balseiro_garcia_naranjo} implies that there exists an exact form $\de \tau$ on $P$ which is a twisting form for $\pi$ (cf. Property \ref{item:11}). Example \ref{exm:regular_twisted_poisson_admit_wtir} shows that the total space of the conormal bundle $ \nu^* \mathcal{F}$ can be endowed with an almost symplectic form $\omega_0$ which makes
  $$ \mathrm{pr}: (\nu^* \mathcal{F}, \omega_0) \to (P,\pi,\de \tau)$$
  \noindent
  into a twisted isotropic realisation of $(P,\pi,\de \tau)$ with non-compact fibres. It can be checked that it is complete in the above sense. Therefore, $(P,\pi,\de \tau)$ is integrable in the sense of \cite{bursztyn_crainic_weinstein_zhu}. This proves that to any bivector $\pi$ whose characteristic distribution is Frobenius integrable can be associated an exact twisting form whose resulting twisted Poisson structure is integrable. For instance, consider $P = S^2 \times \R$; the 2-form
  $$ \tau = (1+t^2) \omega_{S^2}, $$
  \noindent
  where $t$ is a global coordinate on $\R$ and $\omega_{S^2}$ denotes the standard symplectic form on $S^2$, defines a Poisson structure on $P$ whose symplectic foliation $\mathcal{F}$ is given by projection onto the second factor $\mathrm{pr}: P \to \R$. Let $\pi$ denote the resulting bivector. It is well-known (cf. \cite{crainic_fernandes_poisson,weinstein}) that $(P,\pi)$ is not integrable as a Poisson manifold; the above argument shows that $(P,\pi,\de \tau)$  is an integrable twisted Poisson manifold. 
\end{remark}

Property \ref{item:9} ensures that the action defined in Lemma \ref{lemma:lie_algebroid_action} can be integrated to an action of the Lie groupoid integrating the Lie algebroid $\mathrm{pr}: \nu^* \mathcal{F} \to P$, which is nothing but $\mathrm{pr}:\nu^*\mathcal{F} \to P$ with fibrewise addition of 1-forms (this is, in fact, a bundle of abelian, simply connected Lie groups). For each $\alpha \in \Gamma(\nu^*\mathcal{F})$, denote by $\varphi^t_{\alpha}$ the flow of the vector field $\pi_{\sigma}^{\#}(p^*\alpha)$. The smooth Lie groupoid action integrating the infinitesimal action of Lemma \ref{lemma:lie_algebroid_action} is defined by
\begin{equation}
  \label{eq:20}
  \begin{split}
    \mu : \nu^*\mathcal{F} \tensor[_{\mathrm{pr}}]{\times}{_{\pi}} M &\to M \\
    (\alpha, m) &\mapsto \varphi^1_{\alpha}(m),
  \end{split}
\end{equation}
\noindent
where
$$ \nu^*\mathcal{F} \tensor[_{\mathrm{pr}}]{\times}{_{\pi}} M :=\{(\alpha,m) \in \nu^*\mathcal{F} \times M \,:\, \mathrm{pr}(\alpha) = p(m) \,\} $$
\noindent
(cf. Section 5.3 of \cite{moerdijk_mrcun} for generalities on Lie groupoid actions). A way to describe the action of equation \eqref{eq:20} is that, for all $x_0 \in P$, $\nu^*_{x_0} \mathcal{F}$ acts smoothly on $p^{-1}(x_0)$ and the action varies smoothly with $x_0$. In light of Remark \ref{rk:why_the_above_proof_works}, for all $x_0 \in P$, the action of $\nu^*_{x_0} \mathcal{F}$ on $p^{-1}(x_0)$ is transitive. Moreover, since $\nu^*_{x_0} \mathcal{F}$ is abelian, the isotropy subgroups of the action at $m, m' \in p^{-1}(x_0)$ can be \emph{canonically} identified. For each $x_0 \in P$, define
$$ \Lambda_{x_0} :=\{ \alpha_{x_0} \in \nu^*_{x_0} \mathcal{F} :\, \exists m \in p^{-1}(x_0) \text{ s.t. } \varphi^1_{\alpha_{x_0}}(m) = m \}. $$
\noindent 
Note that $p^{-1}(x_0) \cong \nu_{x_0}^*\mathcal{F}/\Lambda_{x_0}$; since $p^{-1}(x_0)$ is compact and $\dim p^{-1}(x_0) = \dim \nu^*_{x_0}\mathcal{F}$, $\Lambda_{x_0}$ is a cocompact discrete subgroup of $\nu_{x_0}^* \mathcal{F}$. Therefore, $p^{-1}(x_0) \cong \mathbb{T}^k$ and $\Lambda_{x_{0}} \cong \Z^k$.

\begin{defn}\label{defn:period_net}
  The subset
  $$ \Lambda : = \coprod_{x_0 \in P} \Lambda_{x_0} \subset \nu^* \mathcal{F} $$
  \noindent
  is called the \emph{period net} associated to the twisted isotropic realisation $\ir$.
\end{defn}

The following theorem gives a local smooth characterisation of twisted isotropic realisations, providing local smooth normal forms, and can be proved using the same arguments as in the Poisson case (cf. Theorem 8.15 in \cite{vaisman}).

\begin{theorem}\label{thm:smooth_structure}
  \mbox{}
  \begin{enumerate}[label=\arabic*), ref=\arabic*]
  \item \label{item:7} The period net $\Lambda \subset \nu^*\mathcal{F}$ is a closed embedded submanifold;
  \item \label{item:10} The composite $\Lambda \hookrightarrow \nu^*\mathcal{F} \to P$ is a fibre bundle whose structure group can be reduced to $\mathrm{GL}(k;\Z)$;
  \item \label{item:14} The quotient $\nu^*\mathcal{F}/\Lambda$ is a smooth manifold which inherits a projection $\mathrm{pr}: \nu^*\mathcal{F}/\Lambda \to P$ which makes it into a bundle of abelian Lie groups (and, hence, a groupoid);
  \item \label{item:15} A choice of locally defined section $s: U \subset P \to M$ of $\ir$ induces a diffeomorphism $ \psi_s: \nu^*\mathcal{F}/\Lambda|_U \to p^{-1}(U) $ which makes the following diagram commutative
    \begin{equation}
      \label{eq:4}
      \xymatrix{\nu^*\mathcal{F}/\Lambda|_U \ar[rr]^-{\psi_s} \ar[dr]_-{\mathrm{pr}} & & p^{-1}(U) \ar[dl]^-{p} \\
        & U. &}
    \end{equation}
    \noindent
    In fact, $\ir$ is a principal $\nu^* \mathcal{F}/\Lambda \to P$-bundle over $P$ and the diffeomorphism $\psi_s$ of equation \eqref{eq:4} is a local trivialisation for $\ir$ (cf. Section 5.7 of \cite{moerdijk_mrcun} and Section 2 of \cite{rossi} for a definition).
  \end{enumerate}
\end{theorem}

\begin{obs} \label{rk:in_local_coordinates}
  The results obtained thus far can also be expressed in terms of local coordinates on $P$. Fix $x_0 \in P$ and let $x^1,\ldots,x^k,x^{k+1},\ldots,x^{2n-k}$ be local coordinates defined near $x_0$ which are adapted to the almost symplectic foliation $\mathcal{F}$ of $\tpm$, so that $x^1,\ldots,x^k$ are local Casimirs. For each $m \in p^{-1}(x_0)$, consider $X_{p^* x^1}(m),\ldots,X_{p^*x^k}(m)$; by Lemma \ref{lemma:lie_algebroid_action}, the $\R$-span of $X_{p^* x^1}(m),\ldots,X_{p^*x^k}(m)$ is an abelian Lie algebra and it equals $T_m p^{-1}(x_0)$. Fix $m_0 \in p^{-1}(x_0)$ and define an $\R^k$-action on $p^{-1}(x_0)$ by
  $$ (t_1,\ldots, t_k) \mapsto \varphi^{t_1}_{x^1} \circ \ldots \circ \varphi^{t_k}_{x^k} (m_0), $$
  \noindent
  where $\varphi^t_{x^j}$ is the flow of $X_{p^*x^j}$ for each $j=1,\ldots,k$. This action is transitive and its isotropy group does not depend on the choice of $m_0$; the isotropy group is precisely $\Lambda_{x_0}$ as defined above. Therefore, $p^{-1}(x_0) \cong \mathbb{T}^k$. Using the implicit function theorem as in \cite{dui} and upon a choice of local section of $\ir$, it can be proved that the isotropy groups $\Lambda_{x_0}$ vary smoothly with $x_0$, thus showing that the identification of the fibres of $p$ with $\mathbb{T}^k$ depends smoothly on the basepoint. This is precisely the content of Theorem \ref{thm:smooth_structure}.
\end{obs}

Theorem \ref{thm:smooth_structure} reduces the smooth classification of twisted isotropic realisations of $\tpm$ with a given period net $\Lambda \to P$ to that of the classification of principal $\nu^* \mathcal{F}/\Lambda \to P$-bundles over $P$. The isomorphism classes of these bundles are in 1-1 correspondence with elements of $\mathrm{H}^1(P;\mathcal{C}^{\infty}(\nu^*\mathcal{F}/\Lambda))$, where $\mathcal{C}^{\infty}(\nu^*\mathcal{F}/\Lambda)$ denotes the sheaf of sections of $\mathrm{pr}: \nu^*\mathcal{F}/\Lambda \to P$ (cf. Section 3.5 of \cite{rossi}). One direction of this correspondence can be seen as follows. Let $\mathcal{U} =\{U_j\}$ be a good open cover of $P$, \textit{i.e.} all finite intersections of elements in $\mathcal{U}$ are contractible. Let $s_j:U_j \to M$ be locally defined sections of $\ir$ and let
$$ \psi_{s_j} : \nu^*\mathcal{F}/\Lambda|_{U_j} \to p^{-1}(U_j) $$
\noindent
denote the local trivialisation given by Theorem \ref{thm:smooth_structure}. Then, for all $j,l$ such that $U_j \cap U_l \neq \emptyset$, the diffeomorphism 
$$ \psi_{s_j}^{-1} \circ \psi_{s_l} : \nu^*\mathcal{F}/\Lambda|_{U_j \cap U_l} \to \nu^*\mathcal{F}/\Lambda|_{U_j \cap U_l} $$
\noindent
is defined by translations along the unique section $\kappa_{jl} + \Lambda : U_j \cap U_l \to \nu^*\mathcal{F}/\Lambda|_{U_j \cap U_l}$ defined by
\begin{equation}
  \label{eq:7}
  \mu(\kappa_{jl},s_l(x)) = s_j(x)
\end{equation}
\noindent
for all $x \in U_j \cap U_l$, where $\mu$ is the action of equation \eqref{eq:20}. Note that the section $\kappa_{jl} : U_j \cap U_l \to \nu^*\mathcal{F}|_{U_j \cap U_l}$ is not uniquely defined, but $\kappa_{jl} + \Lambda$ is, since $\Lambda \to P$ is the isotropy subgroupoid of the action $\mu$. By construction, $\{\kappa_{jl} + \Lambda\}$ form a \v Cech 1-cocycle whose cohomology class in $\mathrm{H}^1(P;\mathcal{C}^{\infty}(\nu^*\mathcal{F}/\Lambda))$ depends only on the isomorphism class of $\ir$. The other direction of the correspondence is explained and used in the proof of Theorem \ref{thm:main}. \\

The short exact sequence of groupoids over $P$ 
$$ \Lambda \hookrightarrow \nu^* \mathcal{F} \to \nu^* \mathcal{F}/\Lambda $$
\noindent 
induces a short exact sequence of sheaves
$$ 0 \to \mathcal{P}_{\Lambda} \to \mathcal{C}^{\infty}(\nu^*\mathcal{F}) \to \mathcal{C}^{\infty}(\nu^*\mathcal{F}/\Lambda) \to 0, $$
\noindent
where $\mathcal{P}_{\Lambda}$ and $\mathcal{C}^{\infty}(\nu^*\mathcal{F})$ denote the sheaves of sections of $\Lambda \to P$ and $\nu^* \mathcal{F} \to P$ respectively. Since $\mathcal{C}^{\infty}(\nu^*\mathcal{F})$ is a fine sheaf (cf. \cite{dui}), the connecting homomorphism
\begin{equation}
  \label{eq:6}
  \delta : \mathrm{H}^1(P;\mathcal{C}^{\infty}(\nu^*\mathcal{F}/\Lambda)) \to \mathrm{H}^2(P;\mathcal{P}_{\Lambda})
\end{equation}
\noindent
(in the long exact sequence in cohomology induced by the above short exact sequence of sheaves) is an isomorphism. Denote the cohomology class corresponding to the isomorphism class of $\ir $ by $\kappa \in \mathrm{H}^1(P;\mathcal{C}^{\infty}(\nu^*\mathcal{F}/\Lambda))$.

\begin{defn}\label{defn:cc}
  The cohomology class $c =\delta(\kappa)$ is called the \emph{Chern class} of the twisted isotropic realisation $\ir$.
\end{defn}

In summary, a twisted isotropic realisation $\ir$ is classified (up to smooth isomorphism) by 
\begin{enumerate}[label=\arabic*), ref=\arabic*]
\item \label{item:23} its period net $\Lambda \subset \nu^*\mathcal{F}$;
\item \label{item:24} its Chern class $c \in \mathrm{H}^2(P;\mathcal{P}_{\Lambda})$, which represents the obstruction to the existence of a section (cf. \cite{daz_delz,dui}).
\end{enumerate}

\begin{remark}\label{rk:standard_understanding_period_net}
  The period net $\Lambda \to P$ associated to a twisted isotropic realisation $\ir$ is a fibre bundle with fibre $\Z^k$ and structure group $\mathrm{GL}(k;\Z)$ and, as such, its isomorphism class is represented by (the conjugacy class of) a representation $\rho: \pi_1(P;x_0) \to \mathrm{GL}(k;\Z)$, for some basepoint $x_0 \in P$. The homomorphism $\rho$ is known in the integrable systems literature as the \emph{Hamiltonian monodromy} of the associated completely integrable Hamiltonian system and is an important invariant of the system, both mathematically and dynamically (cf. \cite{delos,dui}).
\end{remark}

\subsection{Almost symplectic constructions}\label{sec:almost-sympl-constr}
Theorem \ref{thm:smooth_structure} establishes the existence of local smooth trivialisations of twisted isotropic realisations of twisted Poisson manifolds upon a choice of local section (cf. diagram \eqref{eq:4}). In fact, the above statement can be refined to construct local action-angle coordinates in the sense of \cite{fasso_sansonetto} just as in the Poisson context (cf. Proposition \ref{prop:almost_symplectic_pullback}). In order to construct these local almost symplectic normal forms, it is important to notice that the action $\mu$ of $\mathrm{pr}:\nu^*\mathcal{F} \to P$ on $M$ along $\ir$ of equation \eqref{eq:20} is not by diffeomorphisms which preserve $\sigma$. 

\begin{proposition}[cf. Proposition 8.16 in \cite{vaisman}] \label{prop:action_not_almost_symplecto}
  Let $\alpha: U \subset P \to \nu^* \mathcal{F}$ be a locally defined 1-form. If $\varphi^t_{\alpha}$ denotes the flow of $\pi^{\#}_{\sigma}(p^*\alpha)$, then
  $$ (\varphi^1_{\alpha})^* \sigma - \sigma = p^* \de \alpha. $$
\end{proposition}
\begin{proof}
  Note that
  \begin{equation*} 
    \begin{split}
      (\varphi^1_{\alpha})^* \sigma - \sigma &=\int\limits_{0}^{1}\frac{\de}{\de t} (\varphi^t_{\alpha})^* \sigma \de t = \int\limits_{0}^{1} (\varphi^t_{\alpha})^*(L_{\pi^{\#}_{\sigma}(p^*\alpha)} \sigma) \de t \\
      &= \int\limits_{0}^{1} (\varphi^t_{\alpha})^*(i(\pi^{\#}_{\sigma}(p^*\alpha))\de \sigma + \de (i(\pi^{\#}_{\sigma}(p^*\alpha))\sigma)) \de t \\
      &= \int\limits_{0}^{1} (\varphi^t_{\alpha})^*(\underbrace{i(\pi^{\#}_{\sigma}(p^*\alpha))(p^*\phi)}_{=0}) + \de (i(\pi^{\#}_{\sigma}(p^*\alpha))\sigma)) \de t \\
      &=\int\limits_{0}^{1} (\varphi^t_{\alpha})^* \de(p^* \alpha) \de t  = \int\limits_{0}^{1} (p \circ \varphi^t_{\alpha})^* \de \alpha \de t \\
      &= \int\limits_{0}^{1} p^* \de \alpha \de t = p^* \de \alpha,
    \end{split}
  \end{equation*}
  \noindent
  where $i(\pi^{\#}_{\sigma}(p^*\alpha))(p^*\phi) = 0$ as $\pi^{\#}_{\sigma}(p^*\alpha) \in \ker D p$ (cf. Lemma \ref{lemma:lie_algebroid_action}), which also implies that $p \circ \varphi^t_{\alpha} = p$ for all $t \in \R$. This concludes the proof.
\end{proof}

\begin{obs}\label{rk:used_full_strength_assumption}
  Note that Proposition \ref{prop:action_not_almost_symplecto} is the first place where the full strength of condition \ref{item:42} is used. In terms of a NCIHS $F: \asm \to \R^{2n-k}$ on almost symplectic manifolds, the above result depends crucially on property \ref{item:32}, namely that the functions whose Hamiltonian vector fields are tangent to the fibres of $F$ are strongly Hamiltonian (cf. \cite{fasso_sansonetto}).
\end{obs}

\begin{corollary}\label{cor:period_net_closed}
  Any locally defined section $\alpha : U \subset P \to \Lambda$ is a closed 1-form.
\end{corollary}
\begin{proof}
  By definition, $\Lambda$ is the isotropy subgroupoid of the action $\mu$ of $\mathrm{pr}: \nu^*\mathcal{F} \to P$ on $M$ along $\ir$. Therefore, for any locally defined section $\alpha : U \to \Lambda$, have that $\varphi^1_{\alpha} = \mathrm{id}$, which implies that $(\varphi^1_{\alpha})^*\sigma = \sigma$. Proposition \ref{prop:action_not_almost_symplecto} implies that $p^* \de \alpha = 0$. Since $p$ is a surjective submersion, this implies that $\de \alpha = 0$ as required.
\end{proof}

Proposition \ref{prop:action_not_almost_symplecto} can be used to understand how the almost symplectic form $\sigma$ on the total space of a twisted isotropic realisation $\ir$ pulls back under the trivialisations of diagram \eqref{eq:4}. This is the content of the next Proposition, which, in terms of a NCIHS $F: \asm \to \R^{2n-k}$, provides the existence of local action-angle coordinates in the sense of \cite{fasso_sansonetto}. Its proof is omitted as it can be adapted from that of Proposition 8.18 in \cite{vaisman}.

\begin{proposition}\label{prop:almost_symplectic_pullback}
  Let $\ir$ be a twisted isotropic realisation, let $\mathcal{F}$ denote the regular almost symplectic foliation of $\tpm$, and let $\Lambda$ denote the associated period net. For any locally defined section $s: U \subset P \to M$, 
  \begin{equation*}
    \psi_s^* \sigma = \omega_0 + \mathrm{pr}^* \circ s^* \sigma,
  \end{equation*}
  \noindent
  where $\jmath: \nu^*\mathcal{F} \to \cotan P$, $q : \nu^*\mathcal{F} \to \nu^*\mathcal{F}/\Lambda$ and $\mathrm{pr}:\nu^*\mathcal{F}/\Lambda \to P$ denote the natural inclusion, quotient and projection respectively, and $\omega_0$ is the closed 2-form on $\nu^*\mathcal{F}/\Lambda$ satisfying $q^*\omega_0 = \jmath^*\Omega_{\mathrm{can}}$.
\end{proposition}

\begin{remark}\label{rk:existence_omega_quotient}
  The 2-form $\omega_0$ on $\nu^*\mathcal{F}/\Lambda$ is induced by $\jmath^* \Omega_{\mathrm{can}}$, since the action by translations of $\Lambda \to P$ (as a groupoid) on $\nu^*\mathcal{F}$ along $\mathrm{pr}: \nu^* \mathcal{F} \to P$ preserves $\jmath^* \Omega_{\mathrm{can}}$, as $\Lambda$ is locally spanned by closed 1-forms. 
\end{remark}

\section{Constructing twisted isotropic realisations using transversal integral affine geometry}\label{sec:some-title-including}
While Section \ref{sec:twist-isotr-real} solves the classification problem for twisted isotropic realisations of twisted Poisson manifolds, this section gives a criterion that allows to \emph{construct} all twisted isotropic realisations over a fixed twisted Poisson manifold $\tpm$. At the core of the construction problem is the \emph{transversal integral affine geometry} of a twisted isotropic realisation $\ir$ (cf. Section \ref{sec:affine-geometry}), which arises because of the associated period net $\Lambda \to P$. Section \ref{sec:constr-isotr-real} contains the main result of this note, which gives a cohomological obstruction to constructing twisted isotropic realisations (cf. Theorem \ref{thm:main}), which generalises the main result of \cite{daz_delz}.

\subsection{Transversal integral affine geometry of twisted isotropic realisations}\label{sec:affine-geometry}
A consequence of Corollary \ref{cor:period_net_closed} is that the almost symplectic foliation $\mathcal{F}$ of $(\pi,\phi)$ admits a transversally integral affine structure.

\begin{defn} (\cite{goldman_hirsch_levitt})\label{defn:tias}
  Let $B$ be a smooth manifold with a regular foliation $\mathcal{F}'$ of corank $d\geq 1$. A \emph{transversally integral affine structure} on $\mathcal{F}'$ is a smooth atlas $\mathcal{A} = \{(U_j,r_j)\}$ of submersions $r_j : U_j \subset P \to \R^{d}$ which locally define $\mathcal{F}'$ such that, for all $j,l$ such that $U_j \cap U_j \neq \emptyset$ and for each component $C \subset U_j \cap U_j$, there exists a transformation
  $$ (A_{jl,C},\mathbf{c}_{jl,C}) \in \mathrm{Aff}^{\Z}(\R^d):= \mathrm{GL}(d;\Z) \ltimes \R^d $$
  \noindent
  which satisfies $(A_{jl,C},\mathbf{c}_{jl,C}) \circ r_l = r_j$ on $C$.
\end{defn}

The following proposition provides an equivalent definition of transversal integral affine structures in terms of discrete subbundles of conormal bundles.

\begin{proposition}\label{prop:tias_equiv_period_net}
  With the notation of Definition \ref{defn:tias}, a regular foliation $\mathcal{F}'$ admits a transversally integral affine structure $\mathcal{A}$ if and only if there exists a subset $\Lambda \subset \nu^*\mathcal{F}'$ satisfying the following properties:
  \begin{enumerate}[label=(PN\arabic*), ref=(PN\arabic*)]
  \item \label{item:3} $\Lambda$ is a closed, embedded submanifold of $\nu^*\mathcal{F}'$;
  \item \label{item:20} the composite $\Lambda \hookrightarrow \nu^* \mathcal{F}' \to P$ is a fibre bundle whose structure group can be reduced to $\mathrm{GL}(d;\Z)$;
  \item \label{item:25} any locally defined section $\alpha : U \subset B \to \Lambda$ is a closed 1-form.
  \end{enumerate}
\end{proposition}
\begin{proof}[Sketch of proof]
  Suppose that $\mathcal{A} = \{(U_j,r_j)\}$ is a transversally integral affine structure on $\mathcal{F}'$. Let $\Lambda_{\R^d} \subset \cotan \R^d$ be the subset defined by
  $$ \Lambda_{\R^d} := \{ (\mathbf{x},\mathbf{p}) \in \cotan \R^d \; : \; \mathbf{p} \in \Z\langle \de x^1,\ldots, \de x^d \rangle \}. $$
  \noindent
  Note that $\Lambda_{\R^d}$ satisfies properties \ref{item:3}, \ref{item:20} and \ref{item:25}, where the foliation on $\R^d$ has a unique leaf. For each $j$, define 
  $$\Lambda_{j} := r^*_j \Lambda_{\R^d} \subset r^* \cotan \R^d = \nu^*\mathcal{F}'|_{U_j};$$
  \noindent
  the condition that $\mathcal{A}$ defines a transversal integral affine structure implies that $\Lambda_{j}|_{U_l} = \Lambda_{l}|_{U_j}$. Therefore the locally defined $\Lambda_j$ patch together to define a subset $\Lambda \subset \nu^* \mathcal{F}'$ which satisfies properties \ref{item:3}, \ref{item:20} and \ref{item:25}, as the $\Lambda_j$ do locally. \\
  
  Conversely, let $\Lambda \subset \nu^* \mathcal{F}'$ satisfy the above properties. Choose a good open cover $\mathcal{U}=\{U_j\}$ of $B$ and, for each $j$, a frame $\alpha^1_j,\ldots,\alpha^d_j$ of $\Lambda|_{U_j} \to U_j$. Fix $j$; property \ref{item:25} implies that, for all $u = 1,\ldots, d$, $\de \alpha^u_j = 0$. Since $U_j$ is contractible, there exist functions $a^u_{j} : U_j \to \R$, for $u=1,\ldots,d$ such that $\de a^u_j = \alpha^u_j$. The map
  $$ r_j :=(a^1_j,\ldots,a^d_j): U_j \to \R^d $$
  \noindent
  is a submersion which locally defines $\mathcal{F}'$. Property \ref{item:20} implies that, for each $j,l$ such that $U_j \cap U_l \neq \emptyset$, there exist $ (A_{jl},\mathbf{c}_{jl}) \in \mathrm{Aff}^{\Z}(\R^d)$ which satisfy $(A_{jl},\mathbf{c}_{jl}) \circ r_l = r_j$, thus proving that $\mathcal{A}=\{(U_j,r_j)\}$ is a transversally integral affine structure.
\end{proof}

\begin{corollary}\label{cor:tias_twisted_ir}
  A twisted isotropic realisation $\ir$ induces a transversal integral affine structure $\mathcal{A}$ on the almost symplectic foliation $\mathcal{F}$ of $\tpm$.
\end{corollary}
\begin{proof}
  This follows from Proposition \ref{prop:tias_equiv_period_net} applied to the period net $\Lambda \to P$ associated to $\ir$.
\end{proof}

\begin{obs}\label{obs:actions_tiac}
  In terms of a NCIHS $F: \asm \to P \subset \R^{2n-k}$ satisfying hypotheses \ref{item:32} -- \ref{item:40}, the existence of a transversal integral affine structure on the almost symplectic foliation $\mathcal{F}$ of $\tpm$ can be seen as the existence of \emph{action} coordinates on $P$ which are transversal to $\mathcal{F}$ and change as in part (ii) of Proposition 4 in \cite{fasso_sansonetto} (cf. Theorem \ref{thm:lma_almost}).
\end{obs}

Henceforth, identify a transversally integral affine structure with its associated \emph{period net} $\Lambda \to P$ constructed as in Proposition \ref{prop:tias_equiv_period_net}. Fix a twisted isotropic realisation $\ir$ and let $\Lambda \to P$ denote its period net. Corollary \ref{cor:tias_twisted_ir} can be used to infer a relation between the Chern class of $\ir$ and the characteristic class of $\tpm$ (cf. Definition \ref{defn:characteristic_class}) which generalises the results in \cite{daz_delz}. Since $\Lambda \to P$ is locally spanned by closed 1-forms, there is a homomorphism of sheaves
\begin{equation*}
  \begin{split}
    \bar{\de} : \mathcal{C}^{\infty}(\nu^*\mathcal{F}/\Lambda) &\to \mathcal{Z}^2(\nu^*\mathcal{F}) \\
    \alpha + \Lambda &\mapsto \de \alpha, 
  \end{split}
\end{equation*}
\noindent
where $\mathcal{Z}^2(\nu^* \mathcal{F})$ denotes the sheaf of closed 2-forms on $P$ which vanish along the leaves of $\mathcal{F}$. The homomorphism $\bar{\de}$ defined above induces a homomorphism of cohomology groups
\begin{equation}
  \label{eq:5}
  \mathcal{D}_{\Lambda} : \mathrm{H}^1(P;\mathcal{C}^{\infty}(\nu^*\mathcal{F}/\Lambda)) \cong \mathrm{H}^2(P;\mathcal{P}_{\Lambda}) \to \mathrm{H}^1(P;\mathcal{Z}^2(\nu^*\mathcal{F})) \cong \mathrm{H}^3_{\mathrm{rel}}(P;\mathcal{F}),
\end{equation}
\noindent
where the last isomorphism is well-known in foliation theory (cf. \cite{el_kacimi_alaoui,vaisman}). The homomorphism of equation \eqref{eq:5} is called the \emph{Dazord-Delzant} homomorphism (cf. \cite{daz_delz}); it can be defined for any manifold $B$ endowed with a regular foliation $\mathcal{F}'$ which is equipped with a transversally integral affine structure.

\begin{proposition}\label{prop:necessary_main_result}
  The Chern class $c \in \mathrm{H}^2(P;\mathcal{P}_{\Lambda})$ of a twisted isotropic realisation $\ir$ satisfies
  $$ \mathcal{D}_{\Lambda}(c) = \xi_{\phi},$$
  \noindent
  where $\xi_{\phi}$ is the characteristic class of $\tpm$.
\end{proposition}
\begin{proof}
  The proof follows the same lines as that of the analogous statement in the case of isotropic realisations of Poisson manifolds (cf. \cite{daz_delz} and Proposition 8.24 in \cite{vaisman}). Let $\mathcal{U} =\{U_j\}$ be a good open cover of $P$, choose sections $s_j:U_j \to p^{-1}(U_j) $ of $\ir$, and let 
  $$ \kappa_{jl} + \Lambda : U_j \cap U_l \to \nu^* \mathcal{F}/\Lambda|_{U_j \cap U_l} $$
  \noindent
  be the \v Cech cocycle defined as in equation \eqref{eq:7}; its cohomology class 
  $$\kappa \in \mathrm{H}^1(P;\mathcal{C}^{\infty}(\nu^*\mathcal{F}/\Lambda))$$
  \noindent
  is the unique such that $\delta(\kappa)=c$ under the isomorphism of equation \eqref{eq:6}. Proposition \ref{prop:action_not_almost_symplecto} implies that 
  $$ (\varphi^1_{\kappa_{jl}})^* \sigma - \sigma = p^* \de \kappa_{jl} $$
  \noindent
  for any $\kappa_{jl} :  U_j \cap U_l \to \nu^* \mathcal{F}|_{U_j \cap U_l}$ which projects to $\kappa_{jl} + \Lambda$. Applying $s^*_l$ to both sides of the above equation, obtain that 
  $$ s_j^* \sigma - s^*_l \sigma = \de \kappa_{jl}, $$
  \noindent
  since $\varphi^1_{\kappa_{jl}} \circ s_l = s_j$ by definition of $\kappa_{jl}$, and $p \circ s_l = \mathrm{id}$. Since $\kappa_{jl}$ vanishes on the leaves of $\mathcal{F}$, $\de \kappa_{jl}$ is a closed 2-form vanishing on the leaves of $\mathcal{F}$. Thus $\{s_j^* \sigma - s^*_l \sigma\}$ defines a cohomology class in $\mathrm{H}^1(P;\mathcal{Z}^2(\nu^*\mathcal{F}))$. Since the cohomology class of $\{\de \kappa_{jl}\}$ equals $\mathcal{D}_{\Lambda}(c)$ by definition, it remains to prove that the cohomology class defined by $\{s_j^* \sigma - s^*_l \sigma\}$ equals $\xi_{\phi}$. Observe that the sheaf of 2-forms vanishing on $\mathcal{F}$ is fine (cf. \cite{vaisman}), which implies that there exist 2-forms $\lambda_j$ defined on $U_j$ vanishing on $\mathcal{F}$ such that 
  $$ s_j^* \sigma - s^*_l \sigma = \lambda_j - \lambda_l. $$
  \noindent
  Note that the above equation implies that the locally defined forms $s_j^*\sigma - \lambda_j$ patch together to yield a globally defined 2-form $\tau'$. The claim is that $\tau'$ extends the almost symplectic forms on the leaves of $\mathcal{F}$. For, on each $U_j$, the form $s_j^*\sigma -\lambda_j$ restricts to the leaves of $\mathcal{F}$ to $s_j^*\sigma|_{\mathcal{F}}$ by definition of $\lambda_j$; moreover property \ref{item:13} implies that $s_j^*\sigma|_{\mathcal{F}}$ equals the almost symplectic form on the leaves of $\mathcal{F}$, which proves the claim. Observe that, for each $j$,
  $$ s^*_j \de \sigma = s^*_j p^* \phi = \phi|_{U_j}, $$
  \noindent 
  where the first equality follows from property \ref{item:42}. Thus $\de \tau' = \phi - \theta$, where $\theta$ is a closed 3-form vanishing on the leaves of $\mathcal{F}$ whose cohomology class in $\mathrm{H}^3_{\mathrm{rel}}(P;\mathcal{F}) \cong \mathrm{H}^1(P;\mathcal{Z}^2(\nu^*\mathcal{F}))$ is given by the \v Cech cocycle $\{\lambda_l - \lambda_j = - \de \kappa_{jl}\}$. Therefore the cohomology class of $- \theta$ equals $[\de \tau' - \phi]$. Since $\tau'$ is a 2-form on $P$ extending the almost symplectic forms on the leaves of $\mathcal{F}$, $[\de \tau' - \phi] = \xi_{\phi}$ and the result follows.
\end{proof}

\subsection{Constructing twisted isotropic realisations}\label{sec:constr-isotr-real}
Proposition \ref{prop:isotropic_implies_regular} and Corollary \ref{cor:tias_twisted_ir} give necessary conditions for a twisted Poisson manifold to admit a twisted isotropic realisation, namely that it is regular and that its almost symplectic foliation admits a transversal integral affine structure. Therefore, given a regular twisted Poisson manifold $\tpm$ admitting a transversal integral affine structure $\Lambda \to P$ on its almost symplectic foliation $\mathcal{F}$, the construction problem of twisted isotropic realisations consists of two steps:

\begin{enumerate}[label= \textbf{Step} \arabic*, ref= \arabic*] 
\item \label{item:21} classify, up to isomorphism, all transversally integral affine structures on $\mathcal{F}$;
\item \label{item:22} for a fixed transversal integral affine structure $\Lambda \to P$, determine which elements of 
  $$\mathrm{H}^2(P;\mathcal{P}_{\Lambda}) \cong \mathrm{H}^1(P;\mathcal{C}^{\infty}(\nu^*\mathcal{F}/\Lambda))$$
  \noindent
  are the Chern class of some twisted isotropic realisation of $\tpm$ whose associated period net is $\Lambda \to P$.
\end{enumerate}

In what follows, only Step \ref{item:22} is studied, as the classification of transversally integral affine structures on a given foliated manifold is beyond the scope of this note. The main theorem of the present work, which generalises the main result in \cite{daz_delz}, gives a criterion for elements in $\mathrm{H}^2(P;\mathcal{P}_{\Lambda})$ to be the Chern class of a twisted isotropic realisation $\ir$ with associated period net $\Lambda$. 

\begin{theorem}\label{thm:main}
  Let $\tpm$ be a regular twisted Poisson manifold with almost symplectic foliation $\mathcal{F}$ and fix a transversally integral affine structure $\Lambda \to P$. An element $c \in \mathrm{H}^2(P;\mathcal{P}_{\Lambda})$ is the Chern class of a twisted isotropic realisation $\ir$ with period net $\Lambda$ if and only if 
  \begin{equation}
    \label{eq:8}
    \mathcal{D}_{\Lambda}(c) = \xi_{\phi},
  \end{equation}
  \noindent
  where $\mathcal{D}_{\Lambda}$ denotes the Dazord-Delzant homomorphism associated to the transversal integral affine structure $\Lambda \to P$ and $\xi_{\phi}$ is the characteristic class of $\tpm$.
\end{theorem}
\begin{proof}
  Proposition \ref{prop:necessary_main_result} proves that the condition of equation \eqref{eq:8} is necessary. It remains to show that it is sufficient.  Suppose that $c$ satisfies equation \eqref{eq:8}; following the notation of the proof of Proposition \ref{prop:necessary_main_result}, let $\kappa \in \mathrm{H}^1(P;\mathcal{C}^{\infty}(\nu^*\mathcal{F}/\Lambda))$ be the unique class such that $\delta(\kappa) = c$, where $\delta$ is the isomorphism of equation \eqref{eq:6}. As in the case of isotropic realisations of Poisson manifolds, it suffices to check that property \ref{item:42} holds for some almost symplectic form on the total space of a principal $\nu^*\mathcal{F}/\Lambda \to P$-bundle over $P$ whose isomorphism class corresponds to $\kappa$. \\

 Fix a good open cover $\mathcal{U}=\{U_j\}$ of $P$ and let $\{\kappa_{jl} +\Lambda\}$ be a \v Cech 1-cocycle with values in $\mathcal{C}^{\infty}(\nu^*\mathcal{F}/\Lambda)$ representing $\kappa$. A principal $\nu^*\mathcal{F}/\Lambda \to P$-bundle over $P$ whose isomorphism class is classified by $\kappa$ can be constructed by gluing the family $\{\nu^*\mathcal{F}/\Lambda|_{U_j} \to U_j\}$ (over the identity on $P$) via
  \begin{equation*}
    \begin{split}
      \psi_{jl}: \nu^*\mathcal{F}/\Lambda|_{U_j \cap U_l} &\to \nu^*\mathcal{F}/\Lambda|_{U_j \cap U_l} \\
      \alpha + \Lambda &\mapsto \alpha+\kappa_{jl} +\Lambda,
    \end{split}
  \end{equation*}
  \noindent
  where $j,l$ are indices such that $U_j \cap U_l \neq \emptyset$. Denote a representative of this isomorphism class by $p: M \to P$ and note that, for each $j$, the bundles $p:M|_{U_j}\to U_j$, $\mathrm{pr}: \nu^*\mathcal{F}/\Lambda|_{U_j} \to U_j$ can be identified. By assumption, the cohomology class of the \v Cech cocycle $\{\de \kappa_{jl}\}$ is $\xi_{\phi} \in \mathrm{H}^1(P;\mathcal{Z}^2(\nu^*\mathcal{F}))$, where the isomorphism $\mathrm{H}^1(P;\mathcal{Z}^2(\nu^*\mathcal{F})) \cong \mathrm{H}^3_{\mathrm{rel}}(P;\mathcal{F})$ of \cite{el_kacimi_alaoui,vaisman} is used tacitly. Since the sheaf of 2-forms on $P$ vanishing on $F$ is fine (cf. \cite{vaisman}), for each $j$, let $\zeta_j$ be a 2-form defined on $U_j$ and vanishing on $\mathcal{F}|_{U_j}$ such that
  \begin{equation}
    \label{eq:9}
    \zeta_l - \zeta_j = \de \kappa_{jl}
  \end{equation}
  \noindent
  (cf. proof of Proposition \ref{prop:necessary_main_result}).
  Fix a 2-form $\tau \in \Omega^2(P)$ extending the almost symplectic forms on the leaves of $\mathcal{F}$ (cf. Property \ref{item:11}), and define, for each $j$,
  $$ \sigma_j : = \omega_0|_{U_j} + p^* (\tau + \zeta_j), $$
  \noindent
  where $p: M|_{U_j} \to U_j$ denotes projection and $\omega_0$ is the 2-form on $\nu^*\mathcal{F}/\Lambda$ which descends from the restriction of the canonical symplectic form on $\cotan P$ to $\nu^* \mathcal{F}$ (cf. Remark \ref{rk:existence_omega_quotient}). By construction, $\sigma_j$ is almost symplectic, since $\omega_0|_{U_j} + p^* \tau$ is and $\zeta_j$ vanishes on the leaves of $\mathcal{F}|_{U_j}$. The claim is that, for all $j,l$ such that $U_j \cap U_l \neq \emptyset$, $\psi_{jl}^* \sigma_j = \sigma_l$. For,
  \begin{equation*}
    \begin{split}
      \psi_{jl}^* \sigma_j &= \psi_{jl}^*(\omega_0|_{U_j}) + p^*(\tau + \zeta_j) \\
      &=\omega_0|_{U_l} + p^* \de \kappa_{jl} + p^*(\tau + \zeta_j) \\
      &= \omega_0|_{U_l} + p^*(\tau + \zeta_l) = \sigma_l,
    \end{split}
  \end{equation*}
  \noindent
  where the first equality follows from the fact that $p \circ \psi_{jl} = p$, the second from the definition of $\omega_0$ and the fact that, for any 1-form $\alpha$, $\alpha^* \Omega_{\mathrm{can}} = \de \alpha$, where $\Omega_{\mathrm{can}}$ is the canonical symplectic form on $\cotan P$, and the third from equation \eqref{eq:9}. Therefore, the locally defined 2-forms $\{\sigma_j\}$ patch together to give a globally defined almost symplectic 2-form $\sigma$ on the total space of $p: M \to P$. Note that, by construction, 
  $$ \de \sigma = p^*(\de \tau - \theta), $$
  \noindent
  where $\theta$ is a closed 3-form on $P$ vanishing on the leaves of $\mathcal{F}$ whose cohomology class in $\mathrm{H}^3_{\mathrm{rel}}(P;\mathcal{F})$ is defined by the \v Cech cocycle $\{\zeta_l - \zeta_j\}$. Since, by assumption, the cohomology class of $\{\zeta_l - \zeta_j = \de \kappa_{jl}\}$ is $\xi_{\phi}$, it follows that $\de \tau - \phi = \theta + \de \eta$, where $\eta$ is a 2-form on $P$ which vanishes on the leaves of $\mathcal{F}$. Set $\sigma' = \sigma - p^* \eta$. The 2-form $\sigma'$ is almost symplectic since $\sigma$ is and $\eta$ vanishes on the leaves of $\mathcal{F}$; moreover, by construction, $\de \sigma' = p^* \phi$. Therefore property \ref{item:42} holds, thus completing the proof of the theorem.
\end{proof}

\begin{example}\label{exm:sphere}
   Let $P = S^2 \times \R$ and denote by $\mathrm{pr}_{S^2}:P \to S^2$ and $\mathrm{pr}_{\R}: P \to \R$ projections onto the first and second factors respectively. The 2-form 
   $$ \tau = (1+t^2) \omega_{S^2}, $$
   \noindent
   where $\omega_{S^2}$ is the standard symplectic form on $S^2$ and $t$ is the standard coordinate on $\R$, defines a smooth bivector field $\pi$ on $P$ which is $\de \tau$-twisted (in fact, it is Poisson as shown in Remark \ref{rk:complete_realisations_integrability}). Projection onto the second factor $\mathrm{pr}_{\R}$ defines the (almost) symplectic foliation $\mathcal{F}$ of $(P,\pi,\de \tau)$; it admits a transversally integral affine structure defined by 
  $$ \Lambda :=\mathrm{pr}_{\R}^*\Lambda_{\R} \subset \nu^*\mathcal{F},$$
  \noindent
  where $\Lambda_{\R} \to \R$ is given by
  $$ \Lambda_{\R} :=\{(t,\mathbf{p}) \in \cotan \R \;|\; \mathbf{p} \in \Z \langle \de t \rangle \} $$
  \noindent
  (cf. proof of Proposition \ref{prop:tias_equiv_period_net}). Note that $\Lambda \to P$ is a trivial bundle, since it is the pull-back of a trivial bundle. Thus $\mathrm{H}^2(P;\mathcal{P}_{\Lambda}) \cong \mathrm{H}^2(P;\Z)$, where the coefficient system $\Z$ is identified with $\Z \langle \mathrm{pr}_{\R}^*\de t \rangle$. Since $P$ is homotopy equivalent to $S^2$, it follows that elements of $\mathrm{H}^2(P;\mathcal{P}_{\Lambda})$ can be represented by
  $$ d [\mathrm{pr}^*_{S^2}\omega_{S^2}] \otimes \mathrm{pr}_{\R}^*\de t, $$
  \noindent
  where $d \in \Z$ and $[.]$ denotes considering de Rham cohomology class. Since $\Lambda \to P$ is trivial, the Dazord-Delzant homomorphism $\mathcal{D}_{\Lambda}$ can be computed explicitly (cf. \cite{daz_delz}), namely
  $$ \mathcal{D}_{\Lambda}(d [\mathrm{pr}^*_{S^2}\omega_{S^2}] \otimes \mathrm{pr}_{\R}^*\de t) = [d (\mathrm{pr}^*_{S^2}\omega_{S^2}) \wedge (\mathrm{pr}^*_{\R}\de t)] \in \mathrm{H}^3_{\mathrm{rel}}(P;\mathcal{F}).$$
  \noindent
  Note that the characteristic class of $(P,\pi,\de \tau)$ vanishes by construction; thus elements in $\mathrm{H}^2(P;\mathcal{P}_{\Lambda})$ are the Chern class of some twisted isotropic realisation of $(P,\pi,\de \tau)$ with period net $\Lambda \to P$ if and only if they lie in $\ker \mathcal{D}_{\Lambda}$. The claim is that $\mathcal{D}_{\Lambda}$ is injective, so that only $0 \in \mathrm{H}^2(P;\mathcal{P}_{\Lambda})$ arises as a Chern class. Suppose that $d (\mathrm{pr}^*_{S^2}\omega_{S^2}) \wedge (\mathrm{pr}^*_{\R}\de t) = \de \upsilon$, for $\upsilon$ a 2-form which vanishes on the leaves of $\mathcal{F}$. Consider the submanifold with boundary $S^2 \times [0,1] \hookrightarrow P$; Stokes' theorem implies that 
  $$\int\limits_{S^2\times[0,1]} \de \upsilon = \int\limits_{\partial (S^2\times[0,1])} \upsilon = 0, $$
  \noindent
  since $\upsilon$ vanishes on the leaves of $\mathcal{F}$. On the other hand,
  $$ \int\limits_{S^2\times[0,1]} d (\mathrm{pr}^*_{S^2}\omega_{S^2}) \wedge (\mathrm{pr}^*_{\R}\de t) = \int\limits_{\partial (S^2\times[0,1])} - d t \mathrm{pr}^*_{S^2} \omega_{S^2} = -d. $$
  \noindent
  Therefore, $d = 0$, which proves the claim.
\end{example}

\section{Conclusion}\label{sec:conclusion}
As mentioned in the Introduction, one of the main reasons behind this work is to understand non-commutative Hamiltonian integrability for dynamical systems defined on twisted Poisson manifolds (cf. \cite{balseiro_garcia_naranjo}). In order to do this, it is first necessary to make sense of what `non-commutative Hamiltonian integrability' (NCHI for short) means in the twisted Poisson setting. In future work, we aim to give a suitable definition of NCHI on twisted Poisson (and, more generally, Dirac) manifolds  and to study the local, semilocal and global theory of such systems in the spirit of \cite{daz_delz,dui} using the contravariant approach of \cite{lgmv}, with a view to apply this theory to the examples studied in \cite{balseiro_garcia_naranjo}.

\section*{Acknowledgements}
We would like to thank Leo Butler, Francesco Fass\`o, Rui Loja Fernandes, Luis Garc\'ia-Naranjo, Andrea Giacobbe and David Mart\'inez-Torres for various discussions essential to this work. D.S. would like to thank the Universit\`a degli Studi di Padova for the hospitality during the period in which part of this work was carried out.

\end{document}